\documentclass{amsart}
\usepackage{amssymb,amsmath,amscd, amsthm, color,graphicx,bbm,stackrel,paralist,xypic,tikz,eufrak}
\usepackage[font=small,labelfont=bf]{caption}
\usepackage[color,matrix,arrow]{xy}

\newtheorem{lemma}{Lemma}[section]
\newtheorem{thm}[lemma]{Theorem}

\newtheorem{cor}[lemma]{Corollary}
\newtheorem{defn}[lemma]{Definition}
\newtheorem{thm*}{Theorem}

\newtheorem{conj}[lemma] {Conjecture}

\newtheorem{innercustomthm}{Theorem}
\newenvironment{customthm}[1]
  {\renewcommand\theinnercustomthm{#1}\innercustomthm}
  {\endinnercustomthm}

\theoremstyle{definition}

\theoremstyle{remark}

\numberwithin{equation}{section}

\newcommand{\homeo}{\stackrel{\cong}{\to}}

\newcommand{\T}{\mathbf{T}}
\renewcommand{\S}{\mathbf{S}}
\newcommand{\Z}{\mathbb{Z}}
\newcommand{\R}{\mathbb{R}}

\renewcommand{\v}{\mathbf{v}}
\newcommand{\w}{\mathbf{w}}

\renewcommand{\P}{\mathcal{P}}
\newcommand{\Q}{\mathcal{Q}}
\newcommand{\so}{\Rightarrow}

\renewcommand{\iff}{\Leftrightarrow}

\newcommand{\N}{\mathbb{N}}
\newcommand{\p}{\mathbf{p}}

\newcommand{\x}{\mathbf{x}}
\newcommand{\h}{\mathbf{h}}

\newcommand{\e}{\mathbf{e}}
\newcommand{\0}{\mathbf{0}}

\newcommand{\C}{\textsf{{C}}}

\renewcommand{\u}{\mathbf{u}}
\newcommand{\y}{\mathbf{y}}

\newcommand{\ep}{\epsilon}
\newcommand{\Etilde}{\widetilde{E}}
\newcommand{\Ebar}{\overline{E}}
\newcommand{\Ehat}{\widehat{E}}

\renewcommand{\h}{\mathbf{h}}
\renewcommand{\k}{\mathbf{k}}

\newcommand{\actson}{\curvearrowright}
\newcommand{\gothG}{\mathfrak{G}}
\newcommand{\gothH}{\mathfrak{H}}
\newcommand{\Qtilde}{\widetilde{\mathcal{Q}}}
\newcommand{\Qbar}{\overline{\mathcal{Q}}}

\newcommand{\Ptilde}{\widetilde{\mathcal{P}}}

\begin{document}

\title[Topological speedups of $\Z^d$-actions]{Topological speedups of $\Z^d$-actions}

\author[A.S.A. Johnson]{Aimee S. A. Johnson$^1$} 

\thanks{$^1$Corresponding author:  Department of Mathematics and Statistics, Swarthmore College, 500 College Ave.
Swarthmore, PA 19081; email:  \texttt{aimee@swarthmore.edu}}

\author[D. McClendon]{ and David M. McClendon$^2$}

\thanks{$^2$Department of Mathematics and Computer Science, Ferris State University, ASC 2021, Big Rapids, MI 49307; email:  \texttt{mcclend2@ferris.edu}}

%\thanks{}

%\subjclass[2010]{Primary 37A20, Secondary 37A15, 37A35}

\date{February 23, 2021}

\maketitle

\begin{abstract}
We study minimal $\mathbb{Z}^d$-Cantor systems and the relationship between their speedups, their collections of invariant Borel measures, their associated unital dimension groups, and their orbit equivalence classes.  In the particular case of minimal $\mathbb{Z}^d$-odometers, we 
show that their bounded speedups must again be odometers but, contrary to the 1-dimensional case, they need not be conjugate, or even isomorphic, to the original.  
\end{abstract}

% =====================================================================
%
% SECTION 1:
%    INTRODUCTION
%
% =====================================================================

\section{Introduction}

Speedups of measurable dynamical systems have been studied since the 1969 work of Neveu \cite{N,Neveu}.  In this context, the object under consideration is a Lebesgue probability space $(X,\mathcal{X},\mu)$ with an ergodic, measure-preserving transformation $T : X \to X$, and by its ``speedup'' we mean a transformation $T^{p}$ where  $p:X \to \Z^+$.  Neveu characterized exactly which functions $p$ lead to $T^p$ being bijective, and proved a version of Abramov's formula relating the entropy of $T^p$ to the entropy of $T$.  In a seminal paper of 1985, Arnoux, Ornstein, and Weiss \cite{AOW} showed that if $T$ and $S$ are any two ergodic, measure-preserving automorphisms, there is a speedup of $T$ that is measurably conjugate to $S$.  In other words, if the integral of $p$ is no object, then one can speed up $T$ to ``look like" $S$.  This (trivial) classification of ergodic transformations up to ``speedup equivalence'' has the same flavor as work of Dye \cite{Dye, Dye2} in which he proved that all ergodic, measure-preserving automorphisms are (measurably) orbit equivalent.  

The work of Arnoux, Ornstein and Weiss was generalized to ergodic compact group extensions by Babichev, Burton, and Fieldsteel in 2011 \cite{BBF}, and to actions of commuting transformations (i.e. ``higher-dimensional'' actions) by the authors in 2014 \cite{JM1} and 2015 \cite{JM2}.

In this paper we consider speedups in the topological category.  This was first done by Ash \cite{Ash} in 2016 when he studied systems of the form $(X,T)$ where $X$ is a minimal Cantor space and $T: X \to X$ is a homeomorphism.  Similar to how the result of \cite{AOW} reflects Dye's Theorem, Ash's results are closely tied to fundamental results about topological orbit equivalence proved by Giordano, Putnam, and Skau \cite{GPS}.  In particular, the last authors showed that orbit equivalence for such an $(X,T)$ is governed by a unital ordered dimension group that can be associated to the system.  They proved two such systems are orbit equivalent if and only if these dimension groups are isomorphic, and also if and only if there is a homeomorphism between the phase spaces of the systems that induces a bijection between their sets of invariant Borel measures.  Ash's work similarly relates these objects to speedups.  For instance, he showed that one minimal Cantor system is a speedup of another if and only if a surjection between their unital ordered dimension groups exists, and this is in turn equivalent to the existence of a homeomorphism between the phase spaces which induces an injection on the sets of invariant Borel measures. 

This work was continued in a 2018 paper by Alvin, Ash, and Ormes \cite{AAO}, with the additional assumption that the speedup function $p$ is bounded.   They studied the family of minimal Cantor systems given by odometers and showed there is little freedom for their speedups: a minimal bounded speedup of an odometer must be a conjugate odometer.

 In this paper, we study these topological notions in the context of actions of $\Z^d$.   Section 2 provides further background for our work.  In Section 3, we relate speedups to invariant measures for the actions and to the orbit equivalence theory of Giordano, Putnam and Skau.  This section culminates with a series of results which we summarize here:

\begin{customthm}{A}  Suppose $(X_1, \T_1)$ is a minimal $\Z^{d_1}$-Cantor system and $(X_2, \T_2)$ is a minimal $\Z^{d_2}$-Cantor system.  If there is a speedup of $\T_1$ conjugate to $\T_2$, then:
\begin{enumerate}
\item there is a homeomorphism $F : X_1 \to X_2$ which induces an injective transformation from the set of $\T_1$-invariant measures to the set of $\T_2$-invariant measures; and
\item there is a surjective group homomorphism from the dimension group of $\T_2$ to the dimension group of $\T_1$ which preserves the positive cones and order units of those groups.\end{enumerate}
\end{customthm}

\begin{customthm}{B}  Suppose $(X_1, \T_1)$ is a $\Z^{d_1}$-odometer and $(X_2, \T_2)$ is a $\Z^{d_2}$-odometer.  If there is a speedup of $\T_1$ conjugate to $\T_2$, then $\T_1$ and $\T_2$ are orbit equivalent.
\end{customthm}

We prove the individual statements of these theorems in Lemma \ref{maintoptheorempart2}, Theorem \ref{mainodometertheorem}, and Theorem \ref{maintoptheorempart1}.

Section 4 addresses a partial converse of Theorem B for the case where $d_2=1$:

\begin{customthm}{C} 
Suppose $(X_1, \T_1)$ is a $\Z^{d_1}$-odometer and $(X_2, T_2)$ is a $\Z$-odometer.  If $\T_1$ and $T_2$ are orbit equivalent, then 
 for any cone $\C \subseteq \Z^{d_1}$, there is a $\C$-speedup $S$ of $\T_1$ that is topologically conjugate to $T_2$. 
\end{customthm}

Finally, in Section 5 we describe what is meant by a ``bounded'' speedup, and we discuss the properties of those speedups in the setting of 
$\Z^d$-odometers.  We prove in Theorems \ref{bddspeedupthm} and \ref{notconjugate}  the following results, which show that in some ways, bounded speedups of $\Z^d$-odometers are similar to the 
one dimensional setting (statement 1) but in other ways they are not (statement 2).

\begin{customthm}{D} 
Suppose $(X, \T)$ is a free $\Z^{d_1}$-odometer.  If $\S : \Z^{d_2} \actson X$ is a minimal bounded speedup of $\T$, then:
\begin{enumerate}
\item  $(X,\S)$ is a free $\Z^{d_2}$-odometer, but
\item $(X,\S)$ is not  necessarily conjugate to $(X, \T)$.
\end{enumerate}
\end{customthm}

We conclude by showing, via Theorem  \ref{orbteqvandspeedups} and Corollary \ref{nospeedup}, that the choice of cone $\C$ impacts 
whether one can obtain a bounded speedup of one $\Z^d$-odometer which is conjugate to a second:

\begin{customthm}{E} 
\color{white} Z
\color{black}
\begin{enumerate}
\item  For any two $\Z^d$-odometers that are continuously orbit equivalent, there is a cone $\C \subseteq \Z^{d}$ such that one of the odometers is conjugate to a 
$\C$-speedup of the other.
\item There exists $\Z^2$-odometers that are continuously orbit equivalent and a cone $\C \subseteq \Z^{2}$ such that no bounded $\C$-speedup of one odometer is conjugate to the other.
\end{enumerate}
\end{customthm}

% =====================================================================
%
% SECTION 2:
%   TERMINOLOGY
%
% =====================================================================

\section{Terminology}

% ----------------------------------------------------------------------------------
% Section 2.1:  
%      Dynamical systems
% ----------------------------------------------------------------------------------

\subsection{Dynamical systems}

We begin with some standard definitions from topological dynamics.  First, given a group $G$ and a topological space $X$, we say that $(X,\T)$ is a \textbf{$G$-action}, and write $\T : G \actson X$, if for every $g \in G$, there is a homeomorphism $\T^g : X \to X$, and these homeomorphisms satisfy $\T^{gh} = \T^g \circ \T^h$ for every $g, h \in G$ and also that $\T^{0}(x) = x$ for all $x \in X$, where $0$  denotes the identity element of $G$.
In this paper, we concern ourselves with actions where $G = \Z^d$ for some $d$, and will henceforth only give definitions in this setting.  However, the ideas presented in this section apply to actions of more general groups as well.

Given $\T : \Z^d \actson X$ and $x \in X$, the \textbf{orbit} of $x$ is the set $\{\T^\v(x) : \v \in \Z^d\}$.  A $\Z^d$-action $(X,\T)$ is called \textbf{free} if, for any $x \in X$, $\T^\v(x) = x$ implies $\v = \0$.  A $\Z^d$-action is called \textbf{minimal} if every orbit is dense in $X$.  A minimal $\Z^d$-action $(X,\T)$ on a Cantor space $X$ is called a \textbf{minimal $\Z^d$-Cantor system}.  Given a minimal $\Z^d$-Cantor system $(X,\T)$, the set of Borel probability measures invariant under each $\T^\v$ is denoted $\mathcal{M}(X,\T)$.  If for $\mu \in \mathcal{M}(X,\T)$, the only Borel sets invariant under every $\T^\v$ have $\mu$-measure $0$ or $1$, then we say $(X,\T)$ is \textbf{ergodic} with respect to $\mu$.  It is well known that $\mathcal{M}(X,\T) \neq \emptyset$; if $\mathcal{M}(X,\T)$ consists of exactly one measure $\mu$, then $(X,\T)$ is ergodic with respect to $\mu$ and we say that $(X,\T)$ is \textbf{uniquely ergodic}; we may indicate such a system by $(X,\T,\mu)$.

% ----------------------------------------------------------------------------------
% Section 2.2:  
%      Equivalence relations
% ----------------------------------------------------------------------------------

\subsection{Equivalence relations on actions}  A natural problem in topological dynamics is to classify systems up to various notions of equivalence, the most natural notion being conjugacy.  Suppose $\T : \Z^d \actson X$ and $\S : \Z^d \actson Y$.  We say $(X,\T)$ and $(Y,\S)$ are \textbf{(topologically) conjugate} if there is a homeomorphism $\Phi : X \to Y$ such that $\Phi \circ \T^\v = \S^\v \circ \Phi$ for all $\v \in \Z^d$.  

Suppose $\T : \Z^2 \actson X$ and one defines $\S$ by ``switching the generators'' of $\T$, i.e. $\S^{(v_1, v_2)}(x) = \T^{(v_2, v_1)}(x)$.  
In general, such a $\T$ and $\S$ are not conjugate,
 but they must be equivalent in the following weaker sense:  if $\T : \Z^d \actson X$ and $\S : \Z^d \actson Y$, we say $(X,\T)$ and $(Y,\S)$ are \textbf{isomorphic} if there is a homeomorphism $\Phi : X \to Y$ and a group isomorphism $\vartheta : \Z^d \to \Z^d$ such that $\Phi \circ \T^\v = \S^{\vartheta(\v)} \circ \Phi$ for all $\v \in \Z^d$.  In this setting, $\vartheta$ must be given by some matrix in $GL_d(\Z)$.

An even weaker notion of equivalence is when $\T$ and $\S$ can be said to have ``the same orbits''.  More precisely, let $\T : \Z^{d_1} \actson X$ and $\S : \Z^{d_2} \actson Y$.  We say $(X,\T)$ and $(Y,\S)$ are \textbf{orbit equivalent} if there is a homeomorphism $\Phi : X \to Y$ (called an \textbf{orbit equivalence}) such that for every $x \in X$,
\[\Phi\left( \bigcup_{\v \in \Z^{d_1}} \T^\v(x)\right) =  \bigcup_{\v \in \Z^{d_2}} \S^\v\left(\Phi(x)\right).\]

If  $\T : \Z^{d_1} \actson X$ and $\S : \Z^{d_2} \actson Y$ are free actions which are orbit equivalent via $\Phi : X \to Y$, then there is a function 
$\h_\Phi : X \times \Z^{d_2} \to \Z^{d_1}$ such that for all $x \in X$ and $\v\in \Z^{d_2}$,
\[\S^\v(\Phi(x)) = \Phi\left(\T^{\h_\Phi(x,\v)}(x)\right)\]
and a function $\h_{\Phi^{-1}} : Y \times \Z^{d_1} \to \Z^{d_2}$ such that for all $y \in Y$ and $\w\in\Z^{d_1}$,
\[\T^\w(\Phi^{-1}(y)) = \Phi^{-1}\left(\S^{\h_{\Phi^{-1}}(y,\w)}(y)\right).\]
The functions $\h_\Phi$ and $\h_{\Phi^{-1}}$ are called the \textbf{orbit cocycles} associated to the orbit equivalence $\Phi$.  

For a general orbit equivalence, the orbit cocycles $\h_\Phi$ and $\h_{\Phi^{-1}}$ may or may not be continuous, but is natural to ask for some sort of continuity.  With this in mind, we say $\T : \Z^{d_1} \actson X$ and $\S : \Z^{d_2} \actson Y$ are \textbf{continously orbit equivalent} if there is an orbit equivalence $\Phi : X \to Y$ whose orbit cocycles $\h_\Phi$ and $\h_{\Phi^{-1}}$ are continuous with respect to the given topologies on $X$ and $Y$, the discrete topologies on $\Z^{d_1}$ and $\Z^{d_2}$, and the product topologies on $X \times \Z^{d_2}$ and $Y \times \Z^{d_1}$.

It is clear that conjugate actions are isomorphic, isomorphic actions are continuously orbit equivalent, and continuously orbit equivalent actions are orbit equivalent.  However, none of these equivalence relations coincide (see \cite{L}, \cite{CM}, \cite{GPS3}).

% ----------------------------------------------------------------------------------
% Section 2.3:  
%      Speedups
% ----------------------------------------------------------------------------------

\subsection{Speedups}

In this paper, we examine a relation on minimal $\Z^d$-Cantor systems coming from speedups.  Speedups were initially studied by Neveu \cite{N}, \cite{Neveu}, although he did not use the terminology ``speedup''.  However, 
the word ``speedup'' came to be used because of its interpretation in the one-dimensional case (i.e. actions of $\Z$).  Essentially, if  $T : X \to X$ is some map, then a ``speedup'' of $T$ is a map $T^p : X \to X$ where $p : X \to \{1, 2, 3, ...\}$.  So $(X,T^p)$ is a system in which points are ``sped up'', i.e. they move forward more quickly than they do under $T$.  

To define what is meant by a speedup of a $\Z^d$-action, it becomes necessary to specify what one means by ``moving forward''.    Toward that end, we make the following definitions:

\begin{defn} A \emph{\textbf{filled cone}} is an open, connected subset of $\mathbb{R}^d$ whose boundary consists of $d$ distinct hyperplanes passing through the origin.  A \textbf{\emph{cone}} is the intersection of a filled cone with  $(\mathbb{Z}^d - \{\mathbf{0}\})$.
\end{defn}  

In particular, notice the zero vector does not belong to any cone.  There are only two cones which are subsets of $\Z$:  $\Z^+ = \{1,2,3,...\}$ 
and $\Z^- = \{..., -3,-2,-1\}$.  %Given a cone $\C$ and any vector $\mathbf{v} \in \mathbb{Z}^d$, set $\C_{\v} = \C \cap (\C + \v)$. 

\begin{defn} Let $\T : \Z^{d_1} \actson X$ and let $d_2 \in\Z^+$.

A \textbf{\emph{cocycle}} for $\T$ is a function $\p : X \times \Z^{d_2} \to \Z^{d_1}$ such that $\p(x,\0) = \0$ for all $x \in X$ and $\p(x, \v) + \p(\T^{\p(x,\v)}(x), \w) = \p(x, \v + \w)$ for all $x \in X$ and all $\v, \w \in \Z^{d_2}$.  

A \textbf{\emph{speedup}} of $(X,\T)$ is an action $\S : \Z^{d_2} \actson X$ where $\S^\v(x) = \T^{\p(x,\v)}(x)$ for some cocycle $\p$ called the \textbf{\emph{speedup cocycle}}.

Given a cone $\C\subseteq\Z^{d_1}$, if 
$(X,\S)$ is a speedup of $(X,\T)$ such that its speedup cocycle $\mathbf{p}$ satisfies $\p(x, \e_j) \in \C$ for all $j \in \{1, ..., d_2\}$ and all $x \in X$, then we say $(X,\S)$ is a $\C$\textbf{\emph{-speedup}} of $(X,\T)$.  Here, and throughout the paper, $\e_j = (0,...,0,1,0,...,0)$ is the $j^{th}$ standard basis vector.
\end{defn}

Observe that if $(X,\S)$ is a $\C$-speedup of $(X,\T)$ with speedup cocycle $\p$, then it follows from the cocycle relation, together with the fact 
that cones are closed under addition, that $\p(x,\v) \in \C$ for all $\v \in [0,\infty)^{d_2} - \{\0\}$.

Suppose $(X,\T)$ and $(Y,\S)$ are orbit equivalent via $\Phi : X \to Y$.  Then the orbit cocycles $\mathbf{h}_\Phi$ and $\mathbf{h}_{\Phi^{-1}}$ are indeed cocycles, and $\h$ can be thought of as a speedup cocycle giving a speedup of $(X,\T)$ (though not necessarily a $\C$-speedup for a particular cone $\C$) which is conjugate to $(Y,\S)$.  However, speedups are not necessarily orbit equivalences:  given $T : \Z \actson X$, the cocycle $p(x,v) =2v$ defines $(X,T^2)$ as a speedup of $(X,T)$.  In many cases, $T^2$ is not conjugate, nor even orbit equivalent, to $T$.

To define a $\C$-speedup of $(X,\T)$ with $\C\subseteq\Z^{d}$, it is sufficient to specify $d$ functions $\p_1, ..., \p_d : X \to \C$ with the property that, for all $i,j \in\{1, ..., d\}$,
\[\p_i(\T^{\e_j}(x)) + \p_j(x) = \p_j(\T^{\e_i}(x)) + \p_i(x).\]
Then by defining $\S^{\e_j} = \T^{\p_j}$ (in other words, defining $\p(x,\e_j) = \p_j(x)$ and extending so that $\p$ is a cocycle), so long as $\S$ acts by homeomorphisms, the $\Z^d$-action $(X,\S)$ will be a speedup of $(X,\T)$.  We say that the speedup so defined is \textbf{generated} by the $\p_1, ..., \p_d$.

We remark that by definition, a speedup of $(X,\T)$ must be an action by homeomorphisms, so for example, given $T : \Z \actson X$ and 
$p : X \times \Z \to \Z$ 
with $p(x,1) = 2$ and $p(T(x),1) = 1$ for some $x$, $p$ cannot be a speedup cocycle for a speedup of $T$, because said speedup would map both 
$x$ and $T(x)$ to $T^2(x)$.  In \cite{N}, Neveu gave conditions on the values of $p(x,1)$ which are necessary and sufficient for a function 
$p : X \times \{1\} \to \{1,2,3,...\}$ to generate a valid speedup cocycle for a $\Z$-action.

In general, the speedup cocycle $\p$ defining a speedup need not be continuous, but it must be Borel if the action being sped up is a free action:

\begin{thm} \label{thm1.3} Let $\T: \Z^{d_1} \actson X$ and $\S : \Z^{d_2} \actson X$.  Suppose $(X,\S)$ is a speedup of $(X,\T)$.  If $(X,\T)$ is free, then the speedup cocycle $\p : X \times \Z^{d_2} \to \Z^{d_1}$ is a Borel function.
\end{thm}
\begin{proof} Fix $\v \in \Z^{d_1}$ and $\w \in \Z^{d_2}$.  We will show that the set 
\[A(\v,\w) = \{x \in X : \p(x, \w) = \v\}\]
is closed, from which it follows that for any $S \subseteq \Z^{d_1}$,
\[
\p^{-1}(S) = \bigcup_{\v \in S} \bigcup_{\w \in \Z^{d_2}} A(\v,\w)
\]
will be $F_\sigma$, meaning $\p$ is Borel.

Let $\{x_j\} \subseteq A(\v,\w)$ be such that $x_j \to x$.  Since $\T^\v$ and $\S^\w$ are both homeomorphisms, we see that $\T^\v(x_j) \to \T^\v(x)$ and $\S^\w(x_j) \to \S^\w(x)$.  Since, by definition of $A(\v,\w)$, we know $\T^\v(x_j) = \S^\w(x_j)$ for all $j$, we see that $\T^\v(x) = \S^\w(x)$.  Since $(X,\T)$ is free, it follows that $\p(x,\w)= \v$, meaning $x \in A(\v,\w)$.  \end{proof}

Speedups of measure-preserving (as opposed to topological) actions of $\Z^d$ were studied in \cite{AOW} and \cite{BBF} (for $d = 1$) and \cite{JM1} and \cite{JM2} for ($d > 1$).  In particular, a main result of \cite{JM1} is a version of Dye's theorem \cite{Dye, Dye2} stating that given any two ergodic measure-preserving actions of $\Z^d$, they are ``speedup equivalent'', in the sense that for any cone $\C \subset \Z^d$, there is a measurable $\C$-speedup of one which is measurably conjugate to the other.  A major aim of this paper is to investigate analogous results in the topological category.

% ----------------------------------------------------------------------------------
% Section 2.4:  
%      Odometers
% ----------------------------------------------------------------------------------

\subsection{Odometers}\label{odometers}  
We will especially consider a well-studied class of minimal Cantor systems called odometers.  These can be defined in a variety of ways; 
the two approaches we review here are a construction due to Cortez \cite{C} and an equivalent characterization given by Giordano, Putnam and Skau \cite{GPS3}. 

For Cortez' construction, we begin by considering any decreasing sequence $\mathfrak{G} = \{G_j\}_{j = 1}^\infty$ of subgroups of $\Z^d$, where each $G_j$ has finite index in $\Z^d$.  For each $j \geq 1$, let $q_j: \Z^d/G_{j+1} \to \Z^d/G_j$ be the quotient map.  Then, define
\begin{align*}
X_\gothG & = \stackrel[\longleftarrow]{}{\lim} (\Z^d/G_j) \\
& = \{(\x_1, \x_2, \x_3, ...) : \x_j \in \Z^d/G_j \textrm{ and } q_j(\x_{j+1}) = \x_j \textrm{ for all }j\}.
\end{align*}
$X_\gothG $ is a topological group (the topology is the product of the discrete topologies on each $\Z^d/G_n$); for each $j \geq 1$ there is a natural coordinate map $\pi_j : X_\gothG \to \Z^d/G_j$.  More importantly, there is a minimal action $\sigma_\gothG : \Z^d \actson X_\gothG$ given by
\[
\sigma_\gothG^\v(\x_1, \x_2, \x_3, ...) = (\x_1 + \v, \x_2 + \v, \x_3 + \v, ...)
\]  
where the sum in the $j^{th}$ component is taken mod $G_j$.

\begin{defn}[Cortez definition of odometer] \label{CortezDef}
A \textbf{\emph{$\Z^d$-odometer}} is any $\Z^d$-action conjugate to one of the form $(X_\gothG, \sigma_\gothG)$ described above, where $\gothG$ is some decreasing sequence of finite-index subgroups of $\Z^d$.
\end{defn}

We remark that $G$-odometers can be defined for any residually finite group $G$ (not just $\Z^d$); for more, see \cite{CP} or \cite{Dow}.

\begin{thm}[Basic properties of odometers] \label{basicodomprops} Let $\gothG = \{G_1, G_2, ...\}$ be a decreasing sequence of finite-index subgroups of $\Z^d$.
\begin{enumerate}
\item So long as $G_j \neq G_{j+1}$ for infinitely many $j$, $(X_\gothG, \sigma_\gothG)$ is a minimal $\Z^d$-Cantor system;
\item $(X_\gothG, \sigma_\gothG)$ is free if and only if $\stackrel[j=1]{\infty}\bigcap G_j = \{\0\}$;
\item $(X_\gothG, \sigma_\gothG)$ is uniquely ergodic with invariant Borel probability measure $\mu_\gothG$ satisfying $\mu_\gothG \left(\pi_j^{-1}(\x + G_j)\right) = [\Z^d : G_j]^{-1}$ for all $j \geq 1$ and all $\x \in \Z^d$.
\end{enumerate}
\end{thm}

We say that a $\Z^d$-odometer is \textbf{product-type} if it is conjugate to a product of $d$ $\Z$-odometers.  Equivalently, this means the odometer is conjugate to some $(X_\gothG, \sigma_\gothG)$, where each group $G_j \in \gothG$ is of the form $G_j = \stackrel[k=1]{d}{\times} (\Z/m_{j,k}\Z)$.

A second characterization of $\Z^d$-odometers, given by Giordano, Putnam and Skau, 
uses Pontryagin duality. 
Given any compact abelian group $K$,  its Pontryagin dual is $\widehat{K}$, 
the set of continuous group homomorphisms from $K$ to the circle $\mathbb{T} = \R/\Z$.   Now let $H$ be a group such that 
$\Z^d \leq H \leq \mathbb{Q}^d$; give both $H$ and $H/\Z^d$ the discrete topology.  
Then
\[H/\Z^d \subseteq \mathbb{Q}^d/\Z^d \subseteq \R^d/\Z^d \cong \mathbb{T}^d,\]
so there is an inclusion map $\rho : H/\Z^d \to \mathbb{T}^d$.  Let $Y_H = \widehat{H/\Z^d}$; using duality we have
$\widehat{\mathbb{T}^d} \cong \widehat{ \R^d/\Z^d} \subseteq Y_H$.  Since $\Z^d \cong \widehat{\mathbb{T}^d}$, we have
$\widehat{\rho} : \Z^d  \to Y_H$.
Thus for every $\v \in \Z^d$, $\widehat{\rho}(\v)$ is a continuous homomorphism from $H/\Z^d$ to $\mathbb{T}$.
We can then define an action $\psi_H : \Z^d \actson Y_H$ by, for each $\v \in \Z^d$,
\[\psi_H^\v(x) = x + \widehat{\rho}(\v),\]
i.e. for any group homomorphism $x : H/\Z^d \to \mathbb{T}$,
\[(\psi_H^\v(x))(\h + \Z^d) = x(\h) e^{2 \pi i (\h \cdot \v)}.\]

\begin{defn}[Giordano-Putnam-Skau definition of odometer]
A \textbf{\emph{$\Z^d$-odometer}} is any $\Z^d$-action conjugate to one of the form $(Y_H, \psi_H)$ described above.  In this setting, we call $H$ the \emph{\textbf{first cohomology group}} of the odometer.
\end{defn}

The reason $H$ is called the ``first cohomology group'' comes from the following ideas first studied by Forrest and Hunton \cite{FH}.  Given a minimal $\Z^d$-Cantor system $(X,\T)$, let $C(X,\Z)$ be the set of continuous functions from $X$ to $\Z$.  $C(X,\Z)$ is a $\Z^d$-module via usual addition and the scalar multiplication $\v \cdot f = f \circ \T^\v$ for $\v \in \Z^d$, $f \in C(X,\Z)$.  In this context, for an odometer $(Y_H, \psi_H)$, $H = H^1(X,\T)$, the first cohomology group of $\Z^d$ with coefficients in the module $C(X,\Z)$.

It turns out that the Cortez and Giordano-Putnam-Skau definitions produce the same class of systems.  More precisely, given any sequence $\gothG = \{G_1, G_2, ...\}$ as in the Cortez definition, define for each $j$,
\[H_j = \{\v \in \R^d : \v \cdot \x \in \Z \textrm{ for all }\x \in G_j\}\]
and set $H = \stackrel[j=1]{\infty}{\bigcup}H_j$.  The odometer $(Y_H, \psi_H)$ is conjugate to $(X_\gothG, \sigma_\gothG)$.  For the reverse direction, given any $H$ with $\Z^d \leq H \leq \mathbb{Q}^d$, define for each $j$, $H_j = \left(\frac 1{j!}\Z^d\right) \cap H$ and set
\[G_j = \{\v \in \R^d : \v \cdot \x \in \Z \textrm{ for all }\x \in H_j\}.\]
This produces a sequence $\gothG = \{G_1, G_2, ...\}$ for which the corresponding odometer $(X_\gothG, \sigma_\gothG)$ is conjugate to $(Y_H, \psi_H)$.

Notice, in the previous paragraph, the ``dual'' relationship between the $G_j$ and the $H_j$ in the two definitions of $\Z^d$-odometer actions.  As we will need notation for this relationship later, we define, for any set $E \subseteq \R^d$, the set $E^*$ by
\[E^* = \{\v \in \R^d : \v \cdot \x \in \Z \textrm{ for all }\x \in E\}.\]

An advantage of the Giordano-Putnam-Skau approach to defining odometers is that the equivalence relations outlined in Section 1.2 can be easily characterized, for $\Z$ and $\Z^2$-odometers, in terms of the first cohomology group of the action:

\begin{thm}[Theorem 1.5, \cite{GPS3}] \label{isomorphismtest} Let $(X,\T)$ be a free $\Z^{d_1}$-odometer whose first cohomology group is $H(\T)$, and let $(Y,\S)$ be a free $\Z^{d_2}$-odometer whose first cohomology group is $H(\S)$. 
\begin{enumerate}
\item If $d_1, d_2 \leq 2$, then $(X,\T)$ and $(Y,\S)$ are conjugate if and only if $d_1 = d_2$ and $H(\T) = H(\S)$.
\item If $d_1, d_2 \leq 2$, then $(X,\T)$ and $(Y,\S)$ are isomorphic if and only if $d_1 = d_2$ and $\alpha(H(\T)) = H(\S)$ for some $\alpha \in GL_{d_1}(\Z)$.
\item If $d_1, d_2 \leq 2$, then $(X,\T)$ and $(Y,\S)$ are continuously orbit equivalent if and only if $d_1 = d_2$ and $\alpha (H(\T)) = H(\S)$ for some $\alpha \in GL_{d_1}(\mathbb{Q})$ with $\det \alpha = \pm 1$.
\item $(X,\T)$ and $(Y,\S)$ are orbit equivalent if and only if the superindex (see \cite{GPS3}) of $H(\T)$ in $\Z^{d_1}$ equals the superindex of $H(\S)$ in $\Z^{d_2}$.  
\end{enumerate}
\end{thm}

% ----------------------------------------------------------------------------------
% Section 2.5:  
%      Towers and K-R partitions
% ----------------------------------------------------------------------------------

\subsection{Towers, refinements and Kakutani-Rohklin partitions}

This section describes some machinery that will be used in subsequent proofs.  First, for any nonnegative integer $h$, let $[h] = \{0, 1, 2, ..., h-1\}$.  Second, for any vector $\h = (h_1, ..., h_d) \in \Z^d$ with $\h \geq 0$ (meaning $h_j \geq 0$ for all $j$), set $[\h] = [h_1] \times [h_2] \times ... \times [h_d]$.

%  Pretowers

\subsubsection{Pretowers and precastles}

A ``pretower'' is simply a rectangular array of disjoint subsets of $X$ of equal measure:

\begin{defn}  
Let $\mu$ be a Borel probability measure on a Cantor space $X$ and let $\h \in \Z^d$ be such that $\h \geq \0$.  A \textbf{\emph{pretower (in $X$)}} is a collection $\{E(\v) : \v \in [\h]\}$ of clopen subsets $E(\v) \subseteq X$, where the sets are pairwise disjoint and all have the same $\mu$-measure.  The vector $\h$ is called the \textbf{\emph{size}} or \textbf{\emph{height}} of the pretower; $d$ is the \textbf{\emph{dimension}} of the pretower; the individual sets $E(\v)$ are called \textbf{\emph{levels}} of the pretower, and $\mu$ is the {\textbf{\emph{pretower measure}}}.  
\end{defn}

\begin{defn}
A \textbf{\emph{precastle (in $X$)}}  is a set of finitely many pretowers in $X$, all having the same dimension and same $\mu$ for their pretower measure, and where the levels of the pretowers are all disjoint from one another. We denote a precastle by $\P = \{E(\alpha, \v) : 1 \leq \alpha \leq t, \v \in [\h(\alpha)]\}$, which indicates that the precastle consists of $t$ many pretowers of respective heights $\h(\alpha)$.
\end{defn}

%Notice that a single pretower can be thought of as a precastle.  

A one-dimensional (i.e. $d=1$) pretower can be subdivided into three disjoint pieces: the base, the top, and the interior.  We define the 
\textbf{base} and \textbf{top} of a one-dimensional precastle to be the sets
\[
\P_{0} = \bigsqcup_{\alpha = 1}^t E(\alpha, 0) \quad \textrm{ and } \quad 
\P_{2} = \bigsqcup_{\alpha = 1}^t E(\alpha, h(\alpha)-1),
\]
while the \textbf{interior} of a precastle is $\P_1 = \stackrel[\alpha = 1]{t}{\bigsqcup} \stackrel[v=1]{h(\alpha)-2}{\bigsqcup} E(\alpha, v)$.
Setting the \textbf{boundary} of a precastle $\P$, denoted $\partial \P$, to be the union of its base and top, we note that  
$\P_1 = \P - \partial \P$ is the set of points in the precastle that are not in its boundary.

% Towers and Castles

\subsubsection{Towers and castles}   First, given clopen subsets $A$ and $B$ of a Cantor space, we write $T : A \homeo B$ if $T$ is a homeomorphism from $A$ to $B$.  If we define an action by homeomorphisms between the levels of a pretower (precastle), the pretower (precastle) becomes a tower (castle).  

\begin{defn}
Let $\h \in \Z^d$ and suppose $\{E(\v) : \v \in [\h]\}$ is a pretower in a Cantor space $X$.  If for each $\v, \w \in [\h]$, there is  $\T^{\w - \v}: E(\v) \homeo E(\w)$ such that:
\begin{compactenum}
\item for every $\v \in [\h]$, $ \T^{\0} : E(\v) \homeo E(\v)$ is the identity map, 
\item $\T^{\x} \circ \T^{\y} = \T^{\x + \y}$ wherever these maps are defined, and
\item each $\T^{\w - \v}$ preserves the pretower measure $\mu$, meaning for any Borel $A\subseteq E(\w)$, $\mu( \T^{-(\w - \v)} )(A) = \mu(A)$,
\end{compactenum}
then we call the pretower a $\T$\textbf{\emph{-tower}}.  
A $\T$\textbf{\emph{-castle}} is a union of finitely many $\T$-towers of the same dimension and with the same $\mu$ as their pretower measure, and all of whose levels are disjoint.  Given a $\T$-castle $\{E(\alpha, \v) : 1 \leq \alpha \leq t, \v \in [\h(\alpha)]\}$, for any $x \in E(\alpha, \v)$ we define the \textbf{\emph{$\T$-column over $x$}} to be $\{\T^\w(x) : \w \in [\h(\alpha)] - \v\}$.  
\end{defn}

%  K-R partitions

\subsubsection{Kakutani-Rohklin partitions}

Castles are closely related to Kakutani-Rohklin partitions:

\begin{defn}
Let $X$ be a Cantor space and let $\T : \Z^d \actson X$.  A \textbf{\emph{Kakutani-Rohklin (K-R) partition}} for $(X,\T)$ is a partition of $X$ into finitely many clopen sets $\{B(j,\v) : 1 \leq j \leq t, \v \in A(j)\}$, where for each $j$, $A(j)$ is a finite subset of $\Z^d$ containing $\0$, such that for each $\v \in A(j)$, $B(j,\v) = \T^\v(B(j,\0))$.
\end{defn}

If $(X,\T)$ is a minimal $\Z^{d}$-Cantor system with $\mu \in \mathcal{M}(X,\T)$, then any Kakutani-Rohklin partition for $(X,\T)$ with each $A(j) = [\h_j]$ for some $\h_j\in \Z^d$  is a $\T$-castle with pretower measure $\mu$.  However, $\T$-castles need not be K-R partitions for a $\Z^d$-action, because it is possible that $\T$ is not defined on every level, e.g. those levels one could think of as being on the boundary of the rectangle $[\h_j]$.

Odometers possess a useful, standard sequence of K-R partitions, described in the following theorem:

\begin{thm}[K-R partitions for odometers] \label{KRpartthm} Given a $\Z^d$-odometer $(X_\gothG, \sigma_\gothG)$ with unique invariant measure $\mu$, there exists a sequence $\{\P_j\}_{j = 1}^\infty$ of partitions of $X$ with the following properties:
\begin{enumerate}
\item each $\P_j$ is a K-R partition for $(X_\gothG, \sigma_\gothG)$ consisting of one rectangular tower;
\item the partitions $\P_j$ refine, i.e. each atom of $\P_{j}$ is a union of atoms of $\P_{j+1}$;
\item the partitions $\P_j$ generate the topology on $X$;
\item each atom of $\P_j$ has measure $[\Z^d : G_j]^{-1}$;
\item $\mu(\partial \P_j) \to 0$ as $j \to \infty$; and
\item the maximum diameter of any atom of $\P_j$ tends to $0$ as $j \to \infty$.
\end{enumerate}
\end{thm}
\begin{proof} 
Let $\gothG = \{G_1, G_2, ...\}$.   For each $j$, define $d$ integers as follows: let 
\[
m_{j,1} = {\mbox{min}}\{ n>0: n \e_1 \in G_j\}.
\]
Since $G_j$ has finite index, such an integer exists.  Then let 
\[
m_{j,2} =  {\mbox{min}}\{ n>0: \exists\, i \textrm{ such that }n\e_2 + G_j = (i,0,...,0) + G_j \}.
\]  
Analogously, for each $k=1,...,d$, define 
\begin{align*}
m_{j,k } & = {\mbox{min}}\{  n>0: \exists \, i_1, i_2, ..., i_{k-1} \textrm{ such that } \\
& \qquad n\e_k + G_j = (i_1,i_2,...,i_{k-1},0,...,0) + G_j \}.
\end{align*}
Defining $\mathbf{m}_j = (m_{j,1}, m_{j,2}, ..., m_{j,d})$, we see that each rectangle $[\mathbf{m}_j]$ contains exactly one representative element from each coset in $\Z^d/G_j$.  

Finally, for each $j \geq 1$ and each $\v \in [\mathbf{m}_j]$, set $B(j,\v) = \pi_j^{-1}(\v + G_j)$, and define $\P_j = \{B(j,\v) : \v \in [\mathbf{m}_j]\}$.  The partitions so defined satisfy the requirements of the theorem.
\end{proof}

%  Refinements

\subsubsection{Refinements}

Let $\mathcal{T} = \{E(\alpha, v) : 1 \leq \alpha \leq t, v \in [h(\alpha)]\}$ be a one-dimensional $T$-castle (these constructions extend to higher dimensions, but we will only use them when $d =1$).  A \textbf{refinement} of $\mathcal{T}$ is another $T$-castle $\mathcal{T}'$, where each level of $\mathcal{T}'$ is a subset of a single $E(\alpha, v)$.  
One way to construct a refinement of $\mathcal{T}$ is to partition
each $E(\alpha, 0)$ into finitely many clopen sets $\P(\alpha)$ = $\{E(\alpha, 0, 1), ..., E(\alpha, 0,s(\alpha))\}$.  
The map $T$ then induces a partition on each $E(\alpha, v)$ in such a way that each $T$-tower in $\mathcal{T}$ is divided into $s(\alpha)$ disjoint $T$-towers.  We call the $\mathcal{T}'$ so obtained a \textbf{castle refinement over} $\{ P(\alpha): \alpha\in\{1,...,t\} \, \}$.

We next describe two specific methods of obtaining such a castle refinement that will be used in the proof of Theorem \ref{homeoimpliesspeedup}.  For the first, 
suppose $\{x_j\}$ is some finite set where the $T$-columns over each $x_j$ are pairwise disjoint.  For each $\alpha$, partition $E(\alpha, 0)$ into finitely many disjoint clopen sets $E(\alpha, 0, i)$ such that $E(\alpha, 0,i)$ intersects at most one $T$-column of an $x_j$.  This yields a refinement of $\mathcal{T}$ so that the $x_j$ are in separate $T$-towers, and we say the resulting $T$-castle is obtained by
 \textbf{separating the $x_j$ into distinct towers}.

For the second, let $\P$ be any finite clopen partition of $X$.  For each $\alpha  \in \{1, ..., t\}$, let $\P(\alpha)$ be the partition of $E(\alpha,0)$ into ``$\P$-names'', i.e. we partition each $E(\alpha, 0)$ into maximal clopen atoms $E(\alpha, 0, 1), ..., E(\alpha, 0, s(\alpha))$, where for every $x \in E(\alpha, 0, i)$  and every $v \in [h(\alpha)]$, the atom of $\P$ to which $T^v(x)$ belongs depends only on $v$ and $i$, and not on $x$.  The resulting $T$-castle is called the \textbf{refinement of $\mathcal{T}$ into pure $\P$-columns}.

% ----------------------------------------------------------------------------------
% Section 2.6:  
%      Dimension groups
% ----------------------------------------------------------------------------------

\subsection{Dimension groups}\label{dimgroups}

We next describe an algebraic object which is a useful tool for studying Cantor minimal systems and that we will use in Theorem \ref{maintoptheorempart1}.  This algebraic object, called a ``dimension group'', was originally introduced by Elliott \cite{E} as an isomorphism invariant of approximately finite algebras.  A different, but equivalent, approach to defining dimension groups, which we follow here, originates in \cite{EHS}.  

\begin{defn} A \textbf{\emph{partially ordered group}} $(G,G^+)$ is a countable abelian group $G$ together with a subset $G^+ \subseteq G$ called the \textbf{\emph{positive cone}}, satisfying:
\begin{enumerate}
\item $G^+ + G^+ \subseteq G^+$;
\item $G^+ - G^+ = G$;
\item $G^+ \cap (-G^+) = \{0\}$.
\end{enumerate}
A partially ordered group is called \textbf{\emph{unperforated}} if for any $g \in G$, $g+g+ \cdots + g \in G^+$ implies $g \in G^+$.  Given a partially ordered group $(G,G^+)$ and $g,h \in G$, we say $g \leq h$ if $h-g \in G^+$ and $g < h$ if $h-g \in G^+ - \{0\}$.
\end{defn}

\begin{defn}
A \textbf{\emph{dimension group}} is an unperforated, partially ordered group $(G,G^+)$ which satisfies the following property (called the \textbf{\emph{Riesz interpolation property}}):
\begin{itemize}
\item Given any $a_1, b_1, a_2, b_2 \in G$ with $a_i \leq b_j$ for all $i,j \in \{1,2\}$, there is $c \in G$ such that $a_i \leq c \leq b_j$ for all $i,j \in \{1,2\}$.
\end{itemize}
\end{defn}

\begin{defn}
Let $(G,G^+)$ be a partially ordered group.  We call $u \in G^+$ an \textbf{\emph{order unit}} if for every $g \in G$ there is $n \in \mathbb{N}$ such that $g \leq nu$.  A dimension group with an order unit is called a \textbf{\emph{unital dimension group}} and is denoted by $(G,G^+,u)$.  
\end{defn}

In this paper, we will deal only with dimension groups which have the additional property that they are ``simple'':

\begin{defn}  Let $(G,G^+)$ be a partially ordered group.  An \textbf{\emph{order ideal}} is a subgroup $J \leq G$ such that
\begin{enumerate}
\item $J = J^+ - J^+$, where $J^+ = J \cap G^+$; and
\item whenever $0 \leq a \leq b$ and $b \in J$, it follows that $a \in J$.
\end{enumerate}
A dimension group is called \textbf{\emph{simple}} if it has no non-trivial order ideals.
\end{defn}

\begin{defn}
Let $(G,G^+,u)$ be a simple, unital dimension group.  A homomorphism $p : G \to \R$ is called a \textbf{\emph{state}} if $p$ is positive (i.e. $p(G^+) \subseteq [0,\infty)$) and $p(u) = 1$.  

An \emph{\textbf{infinitesimal}} on $(G,G^+,u)$ is an element $a \in G$ such that $p(a) = 0$ for every state $p$ of $G$.  The subgroup consisting of all infinitesimals on $(G,G^+,u)$ is denoted {\rm{Inf ($G$)}}.
\end{defn}

Observe that the quotient group $G/{\rm{Inf }}(G)$ has a natural induced ordering, i.e. $[g] >0$ if $g>0$, where $[g]$ is the coset of $g\in G$.  If $G$ has a distinguished order unit $u$, then $G/{\rm{Inf }}(G)$ inherits the distinguished order unit $[u]$.  Thus if $(G,G^+,u)$ is a unital dimension group then 
$(G/{\rm{Inf }}(G), (G/{\rm{Inf }}(G))^+,[u])$ is also a unital dimension group which has no infinitesimals other than the coset of $0$.

A simple, unital dimension group $G$ always has at least one state, and the states determine the order structure of the dimension group, in that
\[G^+ = \{g \in G : p(g) > 0 \textrm{ for all states }p\textrm{ of }G\} \cup \{0\}.\]
(This is essentially Corollary 4.2 of \cite{Ef}.)

% ----------------------------------------------------------------------------------
% Section 2.7:  
%      Dimension groups related to dynamical systems
% ----------------------------------------------------------------------------------

\subsection{Dimension groups and dynamical systems}

Having laid out the abstract definition of a dimension group, we now turn to the connections between such objects and dynamics (for additional references, see \cite{HPS}, \cite{GPS}, \cite{D}, \cite{GMPS}).  To get started, given a minimal $\Z^d$-Cantor system $(X,\T)$, let $C(X,\Z)$ denote the collection of all continuous integer-valued functions on $X$.  Under addition, $C(X,\Z)$ forms a countable abelian group.  Next, define the set of coboundaries in $C(X,\Z)$, denoted $B_\T$, to be the subgroup of $C(X,\Z)$ generated by all functions of the form $f - f \circ \T^\v$ where $f \in C(X,\Z)$ and $\v \in \Z^d$.  Then set
\[
K^0(X,\T) = C(X,\Z) / B_\T.
\]
We define an ordering on $K^0(X,\T)$ by decreeing that a coset $f + B_\T \in K^0(X,\T)$ belongs to $K^0(X,\T)^+$ precisely when there is a $g \in f + B_\T$ such that $g(x) \geq 0$ for all $x \in X$.  We then have:

\begin{thm}[\cite{For}, Theorem 1.4] 
Let $(X,\T)$ be a Cantor minimal $\Z^d$-system.  Then $(K^0(X,\T), K^0(X,\T)^+,1 + B_\T)$ is a simple, unital dimension group.
\end{thm}

By the comments at the end of Section \ref{dimgroups}, we then also have:

\begin{cor} 
Let $(X,\T)$ be a Cantor minimal $\Z^d$-system and let 
\[G(X,T) = K^0(X,\T)/ {\rm{Inf }}(K^0(X,\T)).\]
Then  $(G(X,\T), G(X,\T)^+,[1 + B_\T])$ is a simple, unital dimension group.
\end{cor}

Giordano, Putnam and Skau \cite{GPS} proved a converse of the preceding theorem, showing that any simple, unital dimension group $(G,G^+,u)$ other than $\Z$ can be realized as $(K^0(X,R),K^0(X,R)^+,1 + B_R)$ for a minimal $\Z$-Cantor system $(X,R)$.  Furthermore, they also showed that orbit equivalence of two minimal $\Z$-Cantor systems $(X,T)$ and $(Y,S)$ corresponds exactly with isomorphism of the systems' associated dimension groups $G(X,T)$ and $G(Y,S)$; this result was extended to actions of $\Z^d$ in \cite{GPS2} and \cite{GMPS}.  

We will see that these ideas carry over to the realm of speedups.  If one minimal $\Z^d$-Cantor system can be sped up to obtain a conjugate version of a second, then one can construct a surjective group homomorphism from the dimension group associated to the second system to the dimension group associated to the first.  This is our upcoming Lemma \ref{maintoptheorempart1}.

The infinitesimal subgroup ${\rm{Inf }}(K^0(X,\T))$ has an alternate characterization which will help us simplify $G(X,\T)$. We first need 
the following relationship between states on the dimension group $K^0(X,\T))$ and $\T$-invariant measures on $X$:

\begin{thm}[\cite{For}, Lemma 7.3] 
\label{statethm}
Let $(X,\T)$ be a Cantor minimal $\Z^d$-system.  Then:
\begin{enumerate}
\item Every $\mu \in \mathcal{M}(X,\T)$ induces a state $p_\mu$ on $(K^0(X,\T), K^0(X,\T)^+,1 + B_\T)$ by
\[p_\mu(f + B_\T) = \int f \, d \mu.\]
\item The map $\mu \mapsto p_\mu$ is a bijective correspondence between $\mathcal{M}(X,\T)$ and the set of states on $(K^0(X,\T), K^0(X,\T)^+,1 + B_\T)$.
\end{enumerate}
\end{thm}

We can then note that the infinitesimals of $(K^0(X,\T), K^0(X,\T)^+,1 + B_\T)$ are exactly the cosets of functions which integrate to $0$ against every $\T$-invariant probability measure.  Defining
\[
Z_\T = \left\{f \in C(X,\Z) : \int f \, d\mu = 0 \textrm{ for all } \mu \in \mathcal{M}(X,\T)\right\},
\]
we can conclude that $ {\rm{Inf}}(K^0(X,\T)) \cong  Z_\T / B_\T$, and finally that
\[
G(X,\T) = K^0(X,\T)/{\rm{Inf}}(K^0(X,\T)) \cong C(X,\Z)/Z_\T,
\]
where the unit of $G(X,\T)$ is exactly the coset $1+Z_\T$.  Note that each state $p$ on $K^0(X,\T)$ then induces a state $\overline{p}$
on $G(X,\T)$ simply by defining, for $h\in K^0(X,\T)$, $\overline{p}(h+{\rm{Inf}}(K^0(X,\T) ) = p(h)$.

% =====================================================================
%
% SECTION 3:
%   SPEEDUPS, INVARIANT MEASURES AND ORBIT EQUIVALENCE
%
% =====================================================================

\section{Speedups, invariant measures and orbit equivalence}

In this section, we explore the relationship between speedups and the associated sets of invariant measures, along with how these relate to orbit equivalence.  We also give a result about speedups and the dimension group.

We begin by comparing the sets of invariant measures for a minimal Cantor system and a minimal speedup of that Cantor system.

%
% LEMMA 3.1 (speedupmeasures)
%   Invariant measures of speedup are superset of invariant measures of original action
%

\begin{lemma}
\label{speedupmeasures}
Let $(X,\T)$ be a minimal $\Z^{d_1}$-Cantor system.  If minimal $\Z^{d_2}$-Cantor system $(X,\S)$ is a speedup of $(X,\T)$, then 
$\mathcal{M}(X,\T) \subseteq \mathcal{M}(X,\S)$.
\end{lemma}
\begin{proof} Let $\mu \in \mathcal{M}(X,\T)$ and let $A \subseteq X$ be Borel.  Fix $\v \in \Z^{d_2}$ and for each $\w \in \Z^{d_1}$, let $A_\w = \{x \in A : \S^\v(x) = \T^\w(x)\}$.  Notice that $\{A_\w : \w \in \Z^{d_1}\}$ is a Borel partition of $A$ and since $\S^\v$ is a homeomorphism, $\{\T^\w(A_\w)  : \w \in \Z^{d_1}\}$ must partition $\S^\v(A)$.  Therefore:
\begin{align*}
\mu(\S^\v(A)) & =  \mu \left(\bigsqcup_{\w \in \Z^{d_1}} \T^\w(A_\w)\right) \\
& = \sum_{\w \in \Z^{d_1}} \mu(\T^\w(A_\w))  \\
& = \sum_{\w \in \Z^{d_1}} \mu(A_\w) \\
& = \mu\left(\bigsqcup_{\w \in \Z^{d_1}} A_\w\right)  \\
& = \mu(A).
\end{align*}
Thus $\mu \in \mathcal{M}(X,\S)$ as desired. 
\end{proof}

We need only slightly modify this to find the relationship between the sets of invariant measures of two minimal Cantor systems when one is conjugate to a speedup of the other.

%
% THEOREM 3.2 (maintoptheorempart2)
%   Speedup of T_1 conjugate to T_2 means injective map from msrs of T_1 to msrs of T_2
%

\begin{lemma} \label{maintoptheorempart2}
Let $(X_1, \T_1)$ be a minimal $\Z^{d_1}$-Cantor system with speedup $(X_1,\S)$ that is a minimal $\Z^{d_2}$-Cantor system.  Suppose $(X_1,\S)$ is conjugate to  $(X_2, \T_2)$.
Then there is a homeomorphism $F : X_1 \to X_2$ such that $F^*: \mathcal{M}(X_1,\T_1) \to \mathcal{M}(X_2,\T_2)$ is injective. 
\end{lemma}
\begin{proof}
Let $F$ be a homeomorphism $F:X_1\rightarrow X_2$ such that $F\circ \S^{\v} = \T_2^{\v}\circ F$, as given to us by the conjugacy between 
$(X_1,\S)$ and  $(X_2, \T_2)$.  We know by Lemma \ref{speedupmeasures} that $\mathcal{M}(X_1,\T_1) \subseteq \mathcal{M}(X_1,\S)$.  Given 
$\mu\in \mathcal{M}(X_1,\T_1)$ and Borel set $A\subseteq X_2$, note that 
\[ 
F^*\left(\mu\right)(T_2^{\v}A) = \mu\left(F^{-1}\left(T_2^{\v}A\right)\right)  = \mu\left(S^{\v}\left(F^{-1} A\right)\right) \\
 =\mu\left(F^{-1} (A)\right)  = F^*\left(\mu\right)(A), 
\]
and thus $F^*$ does send measures in $\mathcal{M}(X_1,\T_1)$ to $\mathcal{M}(X_2,\T_2)$.  If 
$\mu, \nu \in  \mathcal{M}(X_1,\T_1)$ are such that $\mu \neq \nu$, then there is some Borel set $A\subseteq X_1$ with $\mu(A)\neq \nu(A)$.  
But then $F^*(\mu)$ and $F^*(\nu)$ give different values to the Borel set $F(A)\subseteq X_2$, showing that $F^*$ is injective.  \end{proof}

If the systems mentioned in Lemma \ref{speedupmeasures} are uniquely ergodic, then we can rephrase that result as follows.

%
% COROLLARY 3.3 (uniqergcor)
%   Invariant measure is unchanged for speedup of uniquely ergodic system
%

\begin{cor} \label{uniqergcor} 
Let $(X,\T,\mu)$ be a uniquely ergodic, minimal $\Z^{d_1}$-Cantor system.  If uniquely ergodic, minimal $\Z^{d_2}$-Cantor system 
$(X,\S,\nu)$ is a speedup of $(X,\T)$, then 
\begin{enumerate}
\item $\mu = \nu$; and
\item $(X,\S)$ and $(X,\T)$ share the same clopen value set, meaning
\[\{\mu(E) : E \subseteq X \textrm{ is clopen}\} = \{\nu(E) : E \subseteq X \textrm{ is clopen}\}. \]
\end{enumerate}
\end{cor}

If, in addition, the systems have one-dimensional actions, then we can say even more:

%
% COROLLARY 3.4 (uniqergcor1)
%   Speedup of uniquely ergodic Z system is orbit equivalent to original
%

\begin{cor} \label{uniqergcor1} 
Let $(X,T)$ be a uniquely ergodic, minimal $\Z$-Cantor system.  If uniquely ergodic, minimal $\Z$-Cantor system 
$(X,S)$ is a speedup of $(X,T)$, then $(X,S)$ and $(X,T)$ are orbit equivalent.
\end{cor}
\begin{proof} 
This follows from Theorem 2.2 in \cite{GPS}. 
\end{proof}

And if the uniquely ergodic, minimal $\Z$-Cantor systems are in fact odometers, we get an even stronger result:

%
% COROLLARY 3.5 (odometercorollary)
%   For Z-odometers, speedup equivalence implies conjugacy
%

\begin{cor} \label{odometercorollary}
Let $(X_1, T_1)$ and $(X_2, T_2)$ be two $\Z$-odometers.  If $(X_2,T_2)$ is conjugate to a speedup of $(X_1,T_1)$, then $(X_1, T_1)$ and $(X_2, T_2)$ are themselves conjugate.
\end{cor}
\begin{proof} $(X_1, T_1)$ and $(X_2, T_2)$ are orbit equivalent by Corollary \ref{uniqergcor1}.  Orbit equivalent $\Z$-odometers are isomorphic, and therefore conjugate, by the rigidity theorem of Boyle and Tomiyama \cite{BT} (see also Corollary 5.9 in \cite{GPS3}).
\end{proof}

However, our main interest is in higher-dimensional odometers.  We summarize what we know thus far for that context.

%
% THEOREM 3.6 (mainodometertheorem)
%   Main speedup theorem on odometers
%

\begin{thm}
\label{mainodometertheorem} 
Let $(X_1, \T_1, \mu_1)$ be a $\Z^{d_1}$-odometer and let $(X_2, \T_2, \mu_2)$ be a $\Z^{d_2}$-odometer.  Let $\C \subseteq \Z^{d_1}$ be any cone.  Then, for the following statements:
\begin{enumerate}
\item There is a $\C$-speedup of $(X_1,\T_1)$ that is conjugate to $(X_2,\T_2)$. 
%\item There is a surjective group homomorphism $\varphi : (G_2, G_2^+,1 + Z_{\T_2}) \to (G_1, G_1^+,1 + Z_{\T_1})$ such that $\varphi(G_2^+) = G_1^+$ and $\varphi(1+Z_{\T_2}) = 1 + Z_{\T_1}$.
\item There is a homeomorphism $F : X_1 \to X_2$ such that $F^*(\mu_1) = \mu_2$.  
\item $(X_1, \T_1)$ and $(X_2, \T_2)$ are orbit equivalent.
\item $(X_1, \T_1)$ and $(X_2, \T_2)$ have the same clopen value sets.
\end{enumerate}
we have the implications (1) $\so$ (2) $\iff$ (3) $\iff$ (4).
\end{thm}
\begin{proof} (1) $\so$ (2) comes from Lemma \ref{maintoptheorempart2}, together with the fact that odometers are uniquely ergodic.  The equivalence of (2), (3) and (4) follows from Corollaries 2.6 and 2.7 of \cite{GMPS}. \end{proof}

In the next section we will prove a partial converse to this theorem, where we must assume $d_2=1$.  But first we recall that
the invariant measures for $(X,\T)$ correspond to states on $K^0(X,\T)$.  These in turn induce states on the dimension group 
$G(X,T) = K^0(X,\T)/{\rm{Inf}}(K^0(X,\T))$ and  thus Lemma \ref{speedupmeasures} suggests a connection between speedups and maps 
between the dimension groups.  The precise nature of this relationship is as follows:

%
% THEOREM 3.7 (maintoptheorempart1)
%   Speedup implies surjective group homo between dimension groups
%

\begin{thm} \label{maintoptheorempart1}
Let $(X_1, \T_1)$ be a minimal $\Z^{d_1}$-Cantor system with dimension group $G_1 = C(X_1,\Z)/Z_{\T_1}$.  Let $(X_2, \T_2)$ be a minimal $\Z^{d_2}-$Cantor system, with dimension group $G_2 = C(X_2,\Z)/Z_{\T_2}$.  %Let $\C \subseteq \Z^d$ be any cone.  

If there is a speedup of $\T_1$ conjugate to $\T_2$, then there is a surjective group homomorphism $\varphi : (G_2, G_2^+,1 + Z_{\T_2}) \to (G_1, G_1^+,1 + Z_{\T_1})$ such that $\varphi(G_2^+) = G_1^+$ and $\varphi(1 + Z_{\T_2}) = 1 + Z_{\T_1}$.
\end{thm}

\begin{proof} Let $\S$ be the speedup of $\T_1$ conjugate to $\T_2$.   Let ${G} = C(X_1,\Z)/Z_{\S}$ so that $({G}, {G}^+,1 + Z_{\S})$ is the unital dimension group associated to $(X_1,\S)$.  The conjugacy between $(X_1,\S)$ and $(X_2, \T_2)$ induces a unital dimension group isomorphism 
$\varphi_1 :  (G_2, G_2^+,1 + Z_{\T_2}) \to ({G},{G}^+,1 + Z_{\S})$.  Define $\varphi_2 : ({G},{G}^+,1 + Z_{\S}) \to (G_1, G_1^+,1+Z_{\T_1})$ by
\[\varphi_2(g + Z_{\S}) = g + Z_{\T_1}.\]
By Lemma \ref{speedupmeasures}, $\mathcal{M}(X,\T_1) \subseteq \mathcal{M}(X,\S)$, so $Z_{\S} \leq Z_{\T_1}$.   Therefore $\varphi_2$ is well-defined and surjective.  The function $\varphi = \varphi_2 \circ \varphi_1$ gives the desired group homomorphism.
\end{proof}

% =====================================================================
%
%  SECTION 4:
%    A CONVERSE OF THEOREM 3.6
%
% =====================================================================

\section{A converse of Theorem \ref{mainodometertheorem}}  In this section, we prove that the converse of (1) $\so$ (2,3,4) of Theorem \ref{mainodometertheorem} holds when $d_2 = 1$.  This will be Theorem \ref{homeoimpliesspeedup}; we first need a pair of preliminary lemmas:

%
% LEMMA 4.1 (GWodom)
%   Can find subset with smaller measure
%

\begin{lemma} \label{GWodom} Let $(X,\T, \mu)$ be a $\Z^d$-odometer.  
\begin{enumerate}
\item Given two disjoint, clopen subsets $A, B \subseteq X$ with $\mu(A) \leq \mu(B)$, there exists a clopen subset $B' \subseteq B$ with $\mu(A) = \mu(B')$.
\item Given two disjoint, clopen subsets $A, B \subseteq X$ with $\mu(A) = \mu(B)$ and any partition of $A$ into clopen subsets $A_1, ..., A_n$, there is a partition of $B$ into clopen subsets $B_1, ..., B_n$ with $\mu(A_j) = \mu(B_j)$ for all $j$.
\end{enumerate}
\end{lemma}
\begin{proof} Let $\{\P_n\}$ denote the sequence of K-R partitions for $(X,\T)$ coming from Theorem \ref{KRpartthm}.  As this sequence generates the topology of $X$, we can write the two clopen sets $A$ and $B$ as unions of atoms from the $\P_n$'s.  But $A$ and $B$ are also closed subsets of the compact set $X$ and thus are themselves compact sets. This means the union of atoms covering $A$ and $B$ have finite subcovers, meaning we can choose $N \in \N$ such that $A$ and $B$ are both unions of atoms of $\P_N$.

Furthermore, since $\mu(A) \leq \mu(B)$, and since $\P_N$ consists of equal-measure atoms, the number of atoms of $\P_N$ whose union is $A$ (say $a$) must be less than or equal the number of atoms of $\P_N$ whose union is $B$.  Choose $a$ of the atoms of $\P_N$ comprising $B$, and let $B'$ be the union of those atoms.  This proves (1), and statement (2) clearly follows.
\end{proof}

%
% LEMMA 4.2 (FirstLemma)
%   Can map A to B with iterate of odometer T
%

\begin{lemma}
 \label{FirstLemma} 
 Fix a cone $\C \subseteq \Z^d$, and suppose that $(X,\T, \mu)$ is a $\Z^d$-odometer.  Given two disjoint, clopen subsets of $A,B \subseteq X$ of equal positive measure, then there is a function $\p : A \to \C$ such that $\T^{\p} : A \stackrel{\cong}{\rightarrow} B$.
 
Furthermore, given $x_A \in A$ and $x_B \in B$, $\p$ can be chosen so that $\T^\p(x_A) \neq x_B$.
\end{lemma}
\begin{proof}
Let $\{\P_n\}$ denote the sequence of K-R partitions for $(X,\T)$ coming from Theorem \ref{KRpartthm}.  As in the proof of Lemma \ref{GWodom}, we can choose $N \in \N$ such that $A$ and $B$ are both unions of atoms of $\P_N$, and since $\mu(A) = \mu(B)$, the number of atoms of $\P_N$ whose union is $A$ must equal the number of atoms of $\P_N$ whose union is $B$.  We can thus write, using the notation of Theorem \ref{KRpartthm},
\[
A = \bigsqcup_{j=1}^m \pi_N^{-1}(\v_j + G_N); \hspace{.2in} B = \bigsqcup_{j=1}^m \pi_N^{-1}(\w_j + G_N), 
\]
where, without loss of generality, $x_A \in \pi_N^{-1}(\v_1 + G_N)$.  For $j = 1$, choose a vector $\p_1 \in  \C \cap (\w_1 - \v_1 + G_N)$ so that $\T^{\p_1}(x_A) \neq x_B$, and for each $j \in \{2, ..., m\}$, choose any vector $\p_j \in \C \cap (\w_j - \v_j + G_N)$.  Then define
$\p : A \to \Z^d$ by setting $\p(x) = \p_j$ whenever $x \in \pi_N^{-1}(\v_j + G_N)$; we have $\T^{\p}$ which maps $A$ homeomorphically to $B$ as wanted. 
\end{proof}

We now come to the aforementioned converse of Theorem \ref{mainodometertheorem}.  Essentially, the proof uses the homeomorphism described in (2) of Theorem \ref{mainodometertheorem} to mimic the K-R sequence of partitions for the $\Z$-odometer in the phase space of the other odometer, and applies Lemma \ref{FirstLemma} to define the speedup on these partition elements.  This argument follows the general framework of  \cite{Ash}, with modifications to allow for higher dimensions.

%
% THEOREM 4.3 (homeoimpliesspeedup)
%   Converse of main odometer theorem:  orbit equiv implies speedup equiv for odometers if d_2 = 1
%

\begin{thm} 
\label{homeoimpliesspeedup}
Let $(X_1, \T_1, \mu_1)$ be a free $\Z^d$-odometer and let $(X_2, T_2, \mu_2)$ be a free $\Z$-odometer.  If $(X_1, \T_1)$ and $(X_2, T_2)$ are orbit equivalent, then for any cone $\C \subseteq \Z^d$, there is a $\C$-speedup $S$ of $\T_1$ that is topologically conjugate to $T_2$. 
\end{thm}

\begin{proof}

Choose a vector $\mathbf{u} \in \C$ and choose any $x_0 \in X_1$; let $x_2 = \T_1^{-\u}(x_0)$.  Let $\{A_{0,n}\}_{n=0}^\infty$ be a nested, decreasing sequence of clopen sets in $X_1$ with $\stackrel[n]{}{\cap} A_{0,n} = \{x_0\}$ and $diam(A_{0,n}) < \frac 1{2^n}$.  Similarly, let $\{A_{2,n}\}_{n=0}^\infty$ be a nested, decreasing sequence of clopen sets in $X_1$ with $\stackrel[n]{}{\cap} A_{2,n} = \{x_2\}$ and $diam(A_{2,n}) < \frac 1{2^n}$. Without loss of generality, assume $A_{0,0} \cap A_{2,0} = \emptyset$ and that $\mu_1(A_{0,n}) = \mu_1(A_{2,n})$ for every $n$.

By the equivalence of (2) and (3) in Theorem \ref{mainodometertheorem}, there is a homeomorphism $F : X_1 \to X_2$ such that $F^*(\mu_1) = \mu_2$.   \\

\textbf{Main idea of the proof:} \\
We want to mimic the structure of $(X_2,T_2)$ on $X_1$, so we begin by considering the sequence of K-R partitions for $(X_2, T_2)$ described in Theorem \ref{KRpartthm}. We denote this sequence of partitions by $\{\P_2(k)\}_{k=1}^{\infty}$; the sets comprising $\P_2(k)$ are labeled as $\{ B(k,v): v \in [h_2(k)] \}$.  Similarly, let $\{\P_1(k)\}_{k=1}^{\infty}$ be the sequence of K-R partitions for $(X_1,\T_1)$ coming from Theorem \ref{KRpartthm} which, we recall, generates the topology of $X_1$.

We construct the speedup $S$ by induction, with $S$ being defined on more and more of $X_1$ at each step.  At the end of the induction, we will have defined $S$ at all points in $X_1$ except for $x_2$ and we will have defined $S^{-1}$ at all points in $X_1$ except for $x_0$; we can then set $S(x_2) = x_0$ to complete the construction. \\

% ---------------------------
% BASE CASE
% ---------------------------
\textbf{Base case:} We begin the induction by considering the base case, which consists of six steps. \\

% Step 1 of base case
% ---------------------------
\emph{Step 1:  choose partition $\mathcal{P}_2(n_0)$ of $X_2$ and copy that partition over to $X_1$ to obtain $\widetilde{Q}_1(0)$.}  Fix $\ep_0$ such that $0<\ep_0< \mu_1(A_{0,0})$, and then choose $n_0>0$ such that the measure of each atom of $\P_2(n_0)$ is less than $\ep_0$.  We move the structure given by $\P_2(n_0)$ on $X_2$ over to $X_1$ by considering, for each $v \in [h_2(n_0)]\}$, the set $\Etilde_1(0,(1,v)) = F^{-1}(B_2(n_0, v))$.  This collection of sets $\{\Etilde_1(0,(1,v)) : v \in [h_2(n_0)]\}$ forms a partition of $X_1$ we call $\Qtilde_1(0)$. 

A remark regarding our notation:  the reason for the extra ``$1$'' in the sets comprising $\Qtilde_1(0)$ is that all these sets are coming from the first (and only) pretower comprising the precastle $\Qtilde_1(0)$.  Later, there will be multiple (pre)towers, but we want to be able to use the same notation; for instance, by $\Etilde_1(k,(\alpha,v))$ we mean the level at height $v$ in the $\alpha^{th}$ tower of the partition of $X_1$ used in the $k^{th}$ induction step. \\

% Step 2 of base case (swapping)
% ------------------------------------------
\emph{Step 2:  ``Swap'' sets in $\Qtilde_1(0)$ to create the partition $\Qbar_1(0)$ of $X_1$, and define sets $F(0)$ and $R(0)$.} We next adjust $\Qtilde_1(0)$ so that the resulting partition $\Qbar_1(0)$ will have 
 $x_q \in (\Qbar_1(0))_{q} \subseteq A_{q,0}$ for each $q \in \{0,2\}$, where we recall that $(\Qbar_1(0))_0$ is the base and $(\Qbar_1(0))_2$ is the top of the pretower $\Qbar_1(0)$.  This adjustment is done by a procedure we will call ``swapping'', and we will repeat this procedure twice, once for $q = 0$ and again for $q = 2$.  
 At the $q=0$ step, we will delete from $\Etilde_1(0,(1,0))$ the points not in the set $A_{(0,0)}$ and then add in points from $A_{(0,0)}$, including $x_0$ if necessary.  More specifically, first define 
 \[
D_0(0) = \Etilde_1(0,(1,0)) -  \, A_{0,0}.
\]
We know that 
\[ 
 \mu_1\left(\Etilde_1(0,(1,0)) \right) < \ep_0 \leq \mu_1\left( A_{0,0} \right),
 \]
 so by Lemma \ref{GWodom}, we can choose a clopen set $C_0(0) \subseteq A_{0,0} -  \Etilde_1(0,(1,0))$ such that
\[\mu_1(C_0(0)) = \mu_1(D_0(0))\]
and such that $x_0 \in C_0(0)$ whenever $x_0 \notin  \Etilde_1(0,(1,0))$.

Basically, we want to swap out $D_0(0)$ for $C_0(0)$.  To do this rigorously, note 
that the set $C_0(0)$ may intersect a variety of the sets $\Etilde_1(0,(1,w))$ from $\Qtilde_1(0)$. 
Use this to partition $C_0(0)$ into clopen sets      
\[
C_0(0,(1,w)) = C_0(0) \cap \Etilde_1(0,(1,w)).
\]
We think of $(1,w)$ as the current location of this set in the first (and only) pretower.

Next, use Lemma \ref{GWodom} to partition $D_0(0)$ into clopen sets $\{D_0(0,(1,w))\}$ such that 
\[
\mu_1 \left(D_0(0,(1, w)) \right) = \mu_1 \left(C_0(0,(1,w)) \right).
\]
We think of the $(1,w)$ as the location this subset of $D_0(0)$ will be moved to.

We now wish to exchange $D_0(0,(1, w))$ and $C_0(0,(1,w))$ for all $w$.  This means, for each $\Etilde_1(0,(1,x))$, some points might be swapped out and others swapped in.  The specifics depend on whether or not $\Etilde_1(0,(1,x)) =  \Etilde_1(0,(1,0))$. If it is, set 
\[
\widehat{E}_1(0,(1,x)) = \left( \Etilde_1(0,(1,x)) \bigcup C_{0}(0) \right) - D_{0}(0) .
\]
Otherwise, set 
\[
\widehat{E}_1(0,(1,x)) = \left(\Etilde_1(0,(1,x))   -  C_{0}(0,(1,x))\right) \cup D_{0}(0,(1,x)) .
\]
After these changes have been made to create the partition $\{\Ehat_1(0,(1,v)) : v \in [h_2(n_0)]\}$, rename the sets in this partition back to $\{\Etilde_1(0,(1,v)) \}$ and repeat for the $q = 2$ step.  That is, let $u$ be such that $x_2\in\Etilde_1(0,(1,u))$ and swap out some 
$D_2(0) \subseteq \Etilde_1(0,(1,h_2(n_0)-1))$ for some $C_2(0) \subseteq A_{2,0} - \Etilde_1(0,(1,h_2(n_0)-1))$, ensuring $x_2 \in C_2(0)$ unless $u = h_2(n_0)-1$.   Note that when performing this swapping procedure for $q = 2$, we will have $\Ehat(0,(1,0)) = \Etilde(0,(1,0))$, as this set is already a subset of $A_{0,0}$, a set that is disjoint from $A_{2,0}$.
 
When this process has been completed, rename the resulting partition 
\[
\Qbar_1(0) = \left\{\Ebar_1(0,(1,v)) : v \in [h_2(n_0)]\right\}
\]
and note that this partition of $X_1$ satisfies 
\[
x_q \in \left(\Qbar_1(0)\right)_q \subseteq A_{q,0}
\]
for all $q \in \{0,2\}$.  \\

Define $F(0) = \emptyset$ (in the inductive steps, we will use $F(k)$ to record which points have been swapped,  but it is not important to keep track of whether or not points have been swapped in the base step).  Also set $R(0) = \emptyset$ (in the inductive steps, $R(k)$ will denote the set of points where the definition of $S_{k-1}$ was altered in the $k^{th}$ step, but in the base step, there is no $S_{-1}$ to alter, so $R(0)$ is trivially empty).  \color{black} This completes the ``swapping'' procedure.  \\

% Step 3 of base case (construct speedup)
% ------------------------------------------------------
\emph{Step 3:  Construct a $\C$-speedup $S_0$ of $\T_1$ on the pretower $\Qbar_1(0)$.}  For each $v \in [h_2(n_0) - 1]$, use Lemma \ref{FirstLemma} to construct a map 
\[
S_0 : \Ebar_1(0,(1,v)) \homeo \Ebar_1(0,(1,v+1))
\]
such that $S_0(x) = \T^{\p_0(x)}(x)$ for some $\p_0 : \stackbin[v=0]{h_2(n_0)-2}{\bigsqcup} \Ebar_1(0,(1,v)) \to \C$.  We want to ensure $x_0$ and $x_2$ do not end up in the same $S_0$-orbit, so toward that end let $x^* = S_0^{h_2(n_0)-1}(x_0)$.  Then, use Lemma \ref{FirstLemma} to construct $S_0 : \Ebar_1(0,(1,h_2(n_0)-2)) \homeo \Ebar_1(0,(1,h_2(n_0)-1))$ such that $S_0(x) = \T^{\p_0(x)}(x)$ for $\p_0$ taking values in $\C$, but with the additional property that $S_0(x^*) \neq x_2$.  This converts the pretower  $\Qbar_1(0)$ into an $S_0$-tower. \\

% Step 4 of base case (refine partition in X_1)
% ----------------------------------------------------------
\emph{Step 4:  Refine $\overline{\Q}_1(0)$.} We make two modifications to the partition $\overline{\Q}_1(0)$, based on constructions described in Section 2.5.4.  %The first \color{black} is done to ensure that we will end up with a sequence of partitions that generate the topology of $X_1$.  To accomplish this, we refine $\overline{\Q}_1(0)$ using $\P_1(0)$ in the following way:   partition each set $\Ebar_1(0, (1,\0))$ into $t(0)$ clopen subsets $ \Ebar_1(0, (\alpha,\0))$ such that for every $x \in \Ebar_1(0,(\alpha,\0))$ and every $\v \in [\h_2(n_0)]$, the atom of $\P_1(0)$ to which $\S_0^\v(x)$ belongs to depends only on $\v$ and $\alpha$ (and not on $x$).  We call this \textbf{refining $\Qbar_1(0)$ by $\P_1(0)$-names}. 
First, separate the two points $x_0$ and $x_2$ into distinct towers, creating an $S_0$-castle with two different towers.  Second, to ensure that we end up with a sequence of partitions that generate the topology of $X_1$, we refine this $S_0$-castle into pure $\P_1(0)$-columns.  We denote the resulting partition, which is a refinement of $\Qbar_1(0)$ into $t(0)$ many $[h_2(n_0)]$-sized $S_0$-towers, by
\[
\Q_1(0) =  \left\{E_1(0,(\alpha,v)) : 1 \leq \alpha \leq t(0),\,  v \in [h_2(n_0)] \right\}.
\]
\\

% Step 5 of base case (refine partition in X_2)
% ---------------------------------------------------------
\emph{Step 5:  Copy the refinement of Step 4 over to $X_2$.}  Next, we ``copy'' the partition  $\Q_1(0)$ over to $X_2$, producing a refinement of $\P_2(n_0)$ we call $\Q_2(0)$.  To accomplish this,  note that 
\[
\mu_2(B_2(n_0,0)) \, = \, \mu_1\left(\widetilde{E}_1(0,(1,0))\right) \, = \,  \mu_1\left(\Ebar_1(0,(1,0))\right).
\]
and 
\[
\Ebar_1(0,(1,0)) = \bigsqcup_{\alpha=1}^{t(0)} E_1(0,(\alpha,0)).
\]
Partition $B_2(n_0,0)$ into disjoint clopen subsets $\{E_2(0,(\alpha,0)) : 1 \leq \alpha \leq t(0)\}$ such that for each $\alpha$, $\mu_1(E_1(0,(\alpha,0))) = \mu_2(E_2(0,(\alpha,0)))$.   Then let $\Q_2(0)$ be the castle refinement of $\P_2(n_0)$ over this partition of $B_2(n_0,0)$.  We denote this refinement by
\[
\Q_2(0) = \left\{E_2(0,(\alpha,v)) : 1 \leq \alpha \leq t(0), \, \, v \in [h_2(n_0)] \right\}
\]
and observe that $\Q_2(0)$ is a $T_2$-castle, i.e. $T_2^{v}(E_2(0,(\alpha,0)) = E_2(0,(\alpha,v))$ for all $1\le\alpha\le t(0)$ and all $v  \in [h_2(n_0)]$. \color{black} \\

% Step 6 of base case (define conjugacy)
% ----------------------------------------------------
\emph{Step 6:  Define a partial set-wise conjugacy $\Phi_0$.}  Finally, define $\Phi_0 : \Q_1(0) \to \Q_2(0)$ by setting $\Phi_0(E_1(0,(\alpha,v))) = E_2(0,(\alpha,v))$.  This map satisfies 
\[
T_2 \circ \Phi_0 \left(\, E_1(0,(\alpha,v)) \, \right) \, = \, \Phi_0\circ S_0  \left( \, E_1(0,(\alpha,v)) \, \right)
\]
for all $\alpha \in \{1, ..., t(0)\}$ and all $v \in [h_2(n_0)-1]$.\\
%This gives a partial set-wise conjugacy from $(Dom(\S_0), \S_0)$ to $(\Phi_0(Dom(\S_0)), \T_2)$. 

This completes the base case of the proof.   Note that we have now defined integers $n_0$ and $t(0)$, along with partitions $\Q_1(0)$ of $X_1$ and $\Q_2(0)$ of $X_2$, subsets $F(0)$ and $R(0)$ of $X_1$, a ``partially-defined speedup'' $S_0$  of $\T_1$ and a set map $\Phi_0$ intertwining $S_0$ and $T_2$.  \\

% ----------------------------------------------------
%  Inductive step
% ----------------------------------------------------
\textbf{Inductive step:}
Let $k \geq 1$ and suppose we have constructed 
\begin{enumerate}
\item nonnegative integers $n_0 < n_1 < ... < n_{k-1}$;
\item positive integers $t(0), t(1), t(2), ..., t(k-1)$;
\item subsets $F(0), F(1), ..., F(k-1)$ of $X_1$ with $\mu_1(F(j)) < 4 \mu_1(A_{0,j})$ for $j \in \{0,1,..., ,k-1\}$;
\item subsets $R(0), R(1), ..., R(k-1)$ of $X_1$;
%\item for each $l \in \{0, ..., k-1\}$ and each $i \in \{1, 2, ..., t(k-1)\}$, a vector $\h(l,i) \in \Z^d$ with $\h(l,i) > \0$;
\item finite clopen partitions $\Q_1(0), \Q_1(1),  ..., \Q_1(k-1)$ of $X_1$ with
\[\Q_1(j) = \{E_1(j,(\alpha, v)) : 1 \leq \alpha \leq t(j),\, v \in [h_2(n_j)]\}\]
and finite clopen partitions $\Q_2(0), \Q_2(1),  ..., \Q_2(k-1)$ of $X_2$ with
\[\Q_2(j) = \{E_2(j,(\alpha,v)) : 1 \leq \alpha \leq t(j), v \in[h_2(n_j)]\}\]
which satisfy, for all $j \in \{0, ..., k-1\}$:
   \begin{enumerate}
   \item $\Q_1(j)$ refines $\P_1(j)$;
   \item $x_q \in \left(\Q_1(j)\right)_q \subseteq A_{q, j}$ for $q = 0, 2$;
   \item $\Q_2(j)$ refines $\P_2(n_j)$;
      \item $\Q_2(j)$ is a $T_2$-castle, i.e. $T_2 :  E_2(j,(\alpha, v)) \stackrel{\cong}{\to}   E_2(j,(\alpha, v + 1))$ for all $\alpha$ and all $v \in [h_2(n_j)-1]$;
   % , i.e.  $\T_2^\v$ maps $E_2(l,i,\j)$ homeomorphically to $E_2(l,i,\j+\v)$, 
   \end{enumerate}
   \item homeomorphisms $S_0, S_1, ..., S_{k-1}$, where 
   \[S_j : \Q_1(j) - \left(\Q_1(j)\right)_2 \to \Q_1(j) - \left(\Q_1(j)\right)_0\]
   so that, for all $j \in \{0,1,..., k-1\}$:
   \begin{enumerate}
   \item $\Q_1(j)$ is an $S_j$-castle;
   \item $S_j$ is a ``partially defined $\C$-speedup'' of $\T_1$, meaning there is a Borel function $\p_j : \Q_1(j) - \left(\Q_1(j)\right)_2 \to \C$ so that $S_j(x) = \T_1^{\p_j(x)}(x)$; 
   \item $x_0$ and $x_2$ are not in the same $S_j$-orbit; 
   \item for every $x$ in the domain of $S_{j-1}$, we have $\p_j = \p_{j-1}$ (i.e. $S_{j}(x) = S_{j-1}(x)$) unless $x \in R_j$; 
   \end{enumerate}
\item and bijections $\Phi_0, \Phi_1, ..., \Phi_{k-1}$ where each $\Phi_j : \Q_1(j) \to \Q_2(j)$ such that
\[
T_2 \circ \Phi_j(E_1(j,(\alpha,v))) = \Phi_j \circ S_j(E_1(j,(\alpha, v)))
\]
for all $\alpha \in \{1, ..., t(j)\}$ and all $v \in[h_2(n_j)-1]$.
\end{enumerate}

As with the base case, the inductive step itself subdivides into six steps.  \\

% Step 1 of induction step
% ----------------------------------------------------
\emph{Step 1:  Choose a partition $\mathcal{P}_2(n_k)$ of $X_2$, refine it with respect to $\Q_2(k-1)$ and copy the refined partition over to $X_1$ to obtain $\widetilde{\Q}_1(k)$.}
Fix $\epsilon_k$ such that $0<\epsilon_k<\mu(A_{0,k})$ and 
\[
8 \epsilon_k \sum_{j=0}^{k-1} \mu_1(A_{0,j}) < \frac 13 \mu_1(A_{0,k}).
\]\color{black} Choose $n_k > n_{k-1}$ such that 
\[ 
\mu_2\left (\partial \mathcal{P}_2(n_k)\right) < \min \left\{\ep_k, \frac 13 \mu_1(A_{0,k})\right\},
\] 
recalling that $\partial \mathcal{P}_2(n_k)$ is the union of the base and top levels of the K-R partition $\P_2(n_k)$.

% the proportion of sets in $\mathcal{P}_2(n_k)$ along its boundary, as given by its association to the rectangle $[{\bf{h}}_2(n_k)]$, is less than $\epsilon_k$.

At the completion of the $(k-1)^{th}$ step, we have a partition $\Q_2(k-1)$ on $X_2$ where
\[
\Q_2(k-1) = \{E_2(k-1,(\alpha,v)) \, : \, 1\leq \alpha \leq t(k-1), \, v \in [h_2(n_{k-1})] \}.
\]
This partition has $t(k-1)$ towers, each of height $[h_2(n_{k-1})]$. 

The first part of step 1 is to refine $\mathcal{P}_2(n_k)$ into pure $\Q_2(k-1)$-columns; this divides $\mathcal{P}_2(n_k)$--which is one tower of height $h_2(n_k)$--into many (say $s(k)$) towers, denoted by
\[ 
{\widetilde{\mathcal{P}}}_2(n_k) = \left\{{\widetilde{B}}_2(k,(\beta,w)) \, : \, 1\leq \beta \le s(k), \, w \in [h_2(n_k)] \, \right \}.
\]
%where for each $\alpha$ and $v$, there exists $\beta$ and $w$ such that $T_2^{v}(x)\in  E_2(k-1,(\beta, w))$ for every $x\in {\widetilde{B}}_2(n_k,(\alpha,0))$, i.e.  ${\widetilde{B}}_2(n_k,(\alpha,v)) \subseteq  E_2(k-1,(\beta,w))$.

We can also think of this first part of step 1 as dividing the towers of $\Q_2(k-1)$ into sub-towers, which are then grouped together to form the $s(k)$ subdivisions of $\P_2(n_k)$ comprising $\Ptilde_2(n_k)$.  We notate this as follows: each set $E_2(k-1,(\alpha,v))$, for $1\leq \alpha \leq t(k-1)$ and  $v \in [h_2(n_{k-1})]$, is subdivided into 
\[
E_2(k-1,(\alpha, v)) = \bigsqcup_{j=1}^{m(\alpha)} E_2(k-1,(\alpha,v),j)
\]
where for each $(\alpha, v)$ and $j$ there exists $\beta = \beta((\alpha, v), j)$ and $w = w((\alpha, v), j)$ such that
\[
E_2(k-1,(\alpha,v),j) =  {\widetilde{B}}_2(k,(\beta,w)).
\]
Importantly, the map $((\alpha, v),j) \mapsto (\beta,w)$ is a bijection, since  $\Ptilde_2(n_k)$ consists of pure $\Q_2(k-1)$-columns. \\

The second part of step 1 is to partition and rearrange  $\Q_1(k-1)$ in an analogous way, creating a partition $\Qtilde_1(k)$ of $X_1$.  More specifically, for every $\alpha$, $1\le\alpha\le t(k-1)$,
use Lemma \ref{GWodom} to partition $E_1(k-1,(\alpha,0))$ into disjoint clopen subsets  $\{E_1(k-1,(\alpha,0),j),  1\leq j \leq m(\alpha)\}$ where
\[ 
\mu_1\left( \, E_1(k-1,(\alpha, 0),j) \, \right) \, = \, \mu_2\left( \, E_2(k-1,(\alpha, 0),j) \, \right).
\]
Then, let $ \{ \, E_1(k-1,(\alpha,v), j)  : 1\leq \alpha \leq t(k-1), 1 \leq j \leq m(\alpha), \, v \in [h_2(n_{k-1})] \, \}$ be the castle refinement of $\Q_1(k-1)$ over these partitions.  Finally, we define 
the partition $\Qtilde_1(k)$ of $X_1$ into $s(k)$ pretowers of height $h_2(n_k)$ as follows:  given $\beta = \beta((\alpha, v),j) \in \{1, ..., s(k)\}$ and $w = w((\alpha,v),j) \in [h_2(n_k)]$, we define $\Etilde_1(k, (\beta, w)) = E_1(k,(\alpha, v), j)$.    We may assume that the two special points $x_0$ and $x_2$ are in two different pretowers of the precastle $\Qtilde_1(k)$; if not, separate them as described in Section 2.5.4. \\

The third part of Step 1 is to make some adjustments to $\Qtilde_1(k)$, ensuring that $x_q \in (\Qtilde_1(k))_q$ for $q = 0, 2$.  To get started with this part, notice that for every level $\Etilde_1(k, (\beta, w))$ of $\Qtilde_1(k)$, there is $v  \in [h_2(n_{k-1})]$ and an $\alpha \in \{1, ..., t(k-1)\}$ such that 
\[
\Etilde_1(k, (\beta, w)) \subseteq E_1(k-1, (\alpha, v)).
\]
Now, consider the pretower of $\Qtilde_1(k)$ containing $x_0$; suppose this pretower is
\[
\{\Etilde_1(k, (\beta_0, w)) : w \in [h_2(n_k)]\}
\]
and that $x_0 \in \Etilde_1(k, (\beta_0, w_0))$.  Let $\gamma_0$ be so that
\[
x_0 \in \Etilde_1(k, (\beta_0, w_0)) \subseteq E_1(k-1, (\gamma_0, 0)).
\]
Now, we ``change the base'' of the tower of $\Qtilde_1(k)$ containing $x_0$ by setting, for each $w \in [h_2(n_k)]$,
 \[
 \widetilde{\Etilde}_1(k, (\beta_0, w)) =  \Etilde_1(k, (\beta_0, w + w_0 \mod h_2(n_k)))
 \]
where by ``$x \mod h$'' we mean the unique integer in $[h]$ which is congruent to $x$ modulo $h$.  Renaming the sets in this tower as $\Etilde_1(k, (\alpha, v))$, we now have that 
\[x_0 \in \left(\Qtilde_1(k)\right)_0.\]

Repeat the procedure described in the preceding paragraph a second time (if necessary) on a different tower in $\Qtilde_1(k)$, ``changing its base'' so that after the alteration, $x_2 \in (\Qtilde_1(k))_2$.   \\

\color{black}

% Step 2 of induction
\emph{Step 2:  ``Swap'' sets in $\widetilde{Q}_1(k)$ to create the partition $\overline{\Q}_1(k)$ and adjust $S_{k-1}$ as needed.}  
We next adjust $\widetilde{\Q}_1(k)$ so that the resulting partition $\overline{\Q}_1(k)$ will have sets such that $(\overline{Q}_1(k))_q \subset A_{q,k}$ for $q = 0, 2$.  This adjustment is done analogously to how it was done in the base case, using the ``swapping" procedure twice, once for $q =0$ and once for $q = 2$.  
  At the $q^{th}$ step, for each $\Etilde_1(k,(\beta,w)) \in (\widetilde{Q}_1(k))_{q} $, we first delete from it the points in $\Etilde_1(k,(\beta,w))$ that are not in the set $A_{q,k}$ and then add in points from $A_{q,k}$ not already used in another set $\Etilde_1(k,(\beta,w')) \in (\widetilde{Q}_1(k))_q $. %, taking care to ensure that the points added to $\Etilde_1(k, (\alpha, j))$ do not belong to $\widehat{R}(k)$.  
  After this alteration, we will end up with a new partition $\overline{\Q}_1(k)$ which has sets of the same measure and configuration as $\widetilde{\Q}_1(k)$, but has an additional property (akin to Property 5(b) from the induction hypothesis) that $x_q \in \left(\overline{\Q}_1(k)\right)_q \subseteq A_{q, k}$.

The details are the same as the base case, but are repeated here to establish the notation needed to adjust $S_{k-1}$.  First, let $q = 0$ and define
\[D_q(k) = \left(\Qtilde_1(k)\right)_q -  \, A_{q,k}.\]
We know 
\[
\mu_1\left(\partial \Qtilde_1(k)\right) = \mu_2\left(\partial \P_2(n_k)\right) < \frac 13 \mu_1(A_{q,k}),
\]
\color{black} so by Lemma \ref{GWodom} we can choose a clopen set 
\[
C_q(k) \subseteq A_{q,k}  - \partial \Qtilde_1(k)\] 
such that
\[
\mu_1(C_q(k)) = \mu_1(D_q(k)).
\]

We next swap out $D_q(k)$ for $C_q(k)$, exactly the way this swapping was done in the base case (so the details are omitted).  After repeating this same swapping procedure for $q = 2$, define

\[
F(k) = \bigcup_{q \in \{0,2\}} \left( C_q(k) \bigcup D_q(k)\right);
\]
so that $F(k)$ is the set of all points which are swapped during this step.  Observe that
\[
\mu_1\left(F(k)\right) \leq 2 \sum_{q \in \{0,2\}} \mu_1\left(A_{q,k}\right) = 4\, \mu_1\left(A_{0,k}\right);
\] 
establishing statement (3) of the induction. \\

When this process has been completed for $q = 0, 2$, rename the resulting partition 
\[
\overline{Q}_1(k) = \{\Ebar_1(k,(\alpha,v)) : 1\leq \alpha \leq  s(k), \, v \in [h_2(n_k)]\}
\]
and note that this partition of $X_1$ satisfies $
x_q \in \left(\overline{\Q}_1(k)\right)_q \subseteq A_{q,k}$
for $q = 0, 2$.  \\

Unlike the base case, we need to adjust the definition of $S_{k-1}$ to account for the swapping that has taken place.

Let
\[
\check{Q}_1(k) = \{\check{E}_1(k,(\alpha,v)) : 1 \leq \alpha \leq r(k), v \in [h_2(n_k)]\}
\]
be the castle refinement of $\Qbar_1(k)$ into pure $\{F(k), F(k)^C\}$-columns.  Fix one pretower in this partition, say 
\[
 \{\check{E}_1(k,(\alpha,v)) :  v \in [h_2(n_k)]\}.
\]
This pretower divides into blocks of length $h_2(n_{k-1})$, each based at some $\check{E}_1(k, (1, w))$ with $w$ a multiple of $h_2(n_{k-1})$.  For those blocks with no levels contained in $F(k)$, the block forms a $S_{k-1}$-tower and we do not alter the definition of
 $S_{k-1}$.  

However, suppose that a block has one or more levels which are subsets of $F(k)$.  For instance, suppose there is exactly one such level, at height $v$. We can think of this block as looking something like the figure below (where the block is arranged sideways to save space).
 Notice that since $\check{E}_1(k, (1,v))$ was swapped from its previous location, the $S_{k-1}$ defined in the prior induction step no longer maps to and/or from this set as wanted:

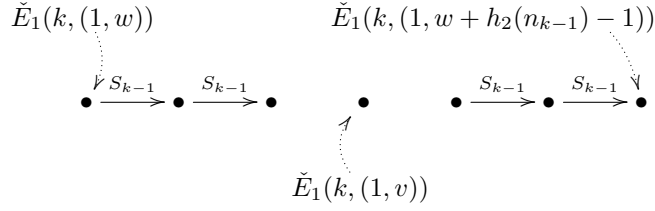
\[
\xymatrix{
{}\save[]+<0cm,0cm>*\txt{%
$\check{E}_1(k,(1,w))$\; }  \ar@{.>}@/^.5pc/[d]
 \restore 
&
&
{}\save[]+<3cm,0cm>*\txt{%
$\check{E}_1(k,(1, w + h_2(n_{k-1}) - 1))$\; } \ar@{.>}@/^1.5pc/[rrrrd]
 \restore 
&
\\
\bullet \ar^{S_{k-1}}[r] 
& \bullet \ar^{S_{k-1}}[r] 
& \bullet 
&  \bullet
& \bullet \ar^{S_{k-1}}[r] 
& \bullet \ar^{S_{k-1}}[r] 
& \bullet \\
&
&
& {}\save[]+<0cm,0cm>*\txt{%
 $\check{E}_1(k,(1,v))$\; }  \ar@{.>}@/^.75pc/[u]
 \restore 
 &
}
\]

Use Lemma \ref{FirstLemma} to (re)define  $S_{k-1} : \check{E}_1(k,(1,v-1)) \homeo \check{E}_1(k,(1,v))$ and/or $S_{k-1} : \check{E}_1(k,(1,v)) \homeo \check{E}_1(k,(1,v+1))$, so that $S_{k-1} = \T_1^{\p(x)}$ for $\p$ taking values in $\C$ (the places where $S_{k-1}$ may have been redefined are indicated by the dashed arrows below).  We remark that if $\check{E}_1(k, (1,v))$ is the base or top of this block, $S_{k-1}$ only needs to be redefined on one set.

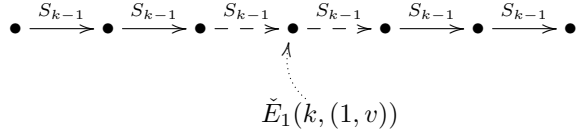
\[
\xymatrix{
\bullet \ar^{S_{k-1}}[r] 
& \bullet \ar^{S_{k-1}}[r]
& \bullet\ar@{-->}^{S_{k-1}}[r]
& \bullet \ar@{-->}^{S_{k-1}}[r] 
& \bullet \ar^{S_{k-1}}[r] 
& \bullet \ar^{S_{k-1}}[r]
& \bullet  \\
& {}\save[]+<3cm,0cm>*\txt{%
$\check{E}_1(k,(1,v))$\; }  \ar@{.>}@/^.75pc/[urr]
 \restore 
& 
%\save "2,5"."3,7"*[F]\frm{} \restore
%\save "2,4"."4,7"*[F]\frm{} \restore
%\save "2,5"."3,7"*[F]\frm{} \restore
%\save "2,1"."6,7"*[F]\frm{} \restore
}
\]

Repeat the procedure outlined above for each level of each block comprising the tower which is a subset of $F(k)$.

Let $R(k)$ be the set of all points in $X_1$ such that $S_{k-1}$ has been redefined via this procedure (this set is the union of the $\check{E}_1(k,(1,v))$s and the $\check{E}_1(k,(1,v-1))$s over the $v$ such that $\check{E}_1(k,(1,v)) \subseteq F(k)$).  This completes the ``swapping'' procedure and the readjustment of $S_{k-1}$.  \\ %As we no further use longer have to worry about the partition $\check{Q}_1(k)$.  \\

% Step 3 of induction case (construct speedup)
\emph{Step 3:  Construct the partial speedup $S_k$ on the pretower $\overline{\Q}_1(k)$.} 
 %, as we see that $\Qbar_1(k)$ is now a partition consisting of $s(k)$ pretowers of size $[h_2(n_k)]$, where sub-blocks of each pretower in $\Qbar_1(k)$ are already towers for $S_{k-1}$.  coming from how the smaller rectangular collections from $\Qtilde_1(k-1)$ of size $[h_2(n_{k-1})]$ were put together to create $\Qtilde_1(k)$ and thus $\Qbar_1(k)$. 
Note that $\Qbar_1(k)$ consists of $s(k)$ pretowers, each of which is made up of blocks of length $h_2(n_{k-1})$ on which $S_{k-1}$ is defined.
 In other words, each pretower consists of unions of smaller $S_{k-1}$-towers, as indicated in the diagram below (again, the tower is presented horizontally).  

\[
\xymatrix{
 \bullet \ar^{S_{k-1}}[r] 
&  \bullet \ar^{S_{k-1}}[r] 
& \bullet \ar^{S_{k-1}}[r]
& \bullet 
& \bullet \ar^{S_{k-1}}[r] 
& \bullet \ar^{S_{k-1}}[r] 
& \bullet \ar^{S_{k-1}}[r]
& \bullet 
}
\]

Notice the heights of the bases of the smaller towers are integers in $[h_2(n_k)]$ that are multiples of $h_2(n_{k-1})$.  

For each $v \in [h_2(n_k)-1]$, define $S_k : \Ebar_1(k, (\alpha, v)) \homeo \Ebar_1(k, (\alpha, v+1))$ to coincide with $S_{k-1}$ if 
$v+1$ is not a multiple of $h_2(n_{k-1})$.  If $v+1$ is a multiple of $h_2(n_{k-1})$, use
Lemma \ref{FirstLemma} so that $S_k = \T_1^{\p(x)}$ for $\p$ taking values in $\C$.  This yields the partial speedup $S_k$ and makes $\overline{Q}_1(k)$ into a $S_k$-castle, establishing (6) of the induction.  \\

% Step 4 of induction case
\emph{Step 4:  Refine $\overline{\Q}_1(k)$.}  Denote by
\[
Q_1(k) = \left\{
 E_1(k,(\alpha,v)): 1 \leq \alpha \leq t(k),  v \in [h_2(n_k)] 
 \right\};
\]
the refinement of $\overline{\Q}_1(k)$ into pure $\P_1(k)$-columns.  This partition has $t(k)$-many $S_k$-towers, each of height $[h_2(n_k)]$.  (5a) of the induction is immediate, and (5b) follows from our work in Step 2 above.  \\

% Step 5 of induction case
\emph{Step 5:  Copy the refinement of Step 4 over to $X_2$.} Recall that we had, for each $\alpha$, 
\[
\mu_2\left({\widetilde{B}}_2(n_k,(\alpha,0))\right) \, = \, \mu_1\left(\Etilde_1(k,(\alpha,0))\right) \, =  \,  \mu_1\left(\Ebar_1(k,(\alpha,0))\right).
\]
Define, for each $\alpha$, the sets $\{ \Ebar_1(k,(\alpha,0),j): 1\le j\le q(\alpha) \}$; these are the atoms of $\Q_1(k)$ contained in the atom $\Ebar_1(k,(\alpha,0))$ which belongs to the base of $\Qbar_1(k)$.  Thus
$\Ebar_1(k, (\alpha,0)) =  \stackrel[j=1]{q(\alpha)}{\bigsqcup} \Ebar_1(k,(\alpha,0),j)$ gives different notation for the refinement constructed in Step 4.  Then choose
disjoint clopen subsets $\{ E_2(k,(\alpha,0),j) \}$, whose union is all of  ${\widetilde{B}}_2(n_k,(\alpha,0))$,  with $\mu_2(E_2(k,(\alpha,0),j) = \mu_1(\Ebar_1(k,(\alpha,0),j)$.  Denote by
\[ 
Q_2(k) = \left\{ E_2(k,(\alpha,v)): \, 1\leq \alpha \leq t(k), \, v \in [h_2(n_k)] \right\}
\] 
the castle refinement of  ${\widetilde{\mathcal{P}}}_2(n_k)$ over these partitions; we now have (5c) and (5d) of the induction.  \\

% Step 6 of induction case 
\emph{Step 6:  Define the partial set-wise conjugacy $\Phi_l$.} Finally, define $\Phi_k : \Q_1(k) \to \Q_2(k)$ by setting $\Phi_k\left(E_1(k,(\alpha,v))\right) = E_2(k,(\alpha, v))$.  
This map satisfies 
\[
T_2 \circ \Phi_k \left( \, E_1(k,(\alpha, v)) \, \right) \, = \, \Phi_k\circ S_k  \left( \, E_1(k,(\alpha, v) ) \, \right)
\]
for all $\alpha\in \{1, ..., t(k)\}$ and all $v \in [h_2(n_k) - 1]$.  This establishes (7) of the induction, and completes the induction step.  \\
\\

\textbf{Conclusion:}   After completing the induction procedure, we have constructed:
\begin{enumerate}
\item subsets $R(0), R(1), R(2), ...$ of $X_1$, where each $R(k)$ is the set of points where $S_k$ does not equal $S_{k-1}$;
%\item for each $l \geq 0$ and each $i \in \{1, 2, ..., t(l)\}$, a positive integer $h(n_l,i)$;

\item finite clopen partitions $\Q_1(0), \Q_1(1), \Q_1(2), ...$ which refine and generate the topology of $X_1$ with
     \[
     \Q_1(k) = \{E_1(k,(\alpha, v)) : 1 \leq \alpha \leq t(k),\, v \in [h_2(n_k)]\}
     \]
     and finite clopen partitions $\Q_2(0), \Q_2(1), \Q_2(2), ...$ which refine and generate the topology of $X_2$ with
     \[
     \Q_2(k) = \{E_2(k,(\alpha, v)) : 1 \leq \alpha \leq t(k),\, v \in[h_2(n_k)]\}
     \]
          which satisfy, for all $k$:
    \begin{enumerate}
           \item $x_q \in \left(\Q_1(k)\right)_q \subseteq A_{q, k}$ for $q = 0, 2$;
          \item $\Q_2(k)$ is a $T_2$-castle;
          \item $\Q_1(k)$ is an $S_k$-castle, where $S_k(x) = \T_1^{\p(x)}(x)$ for $\p$ taking values in $\C$ and $x_0$ and $x_2$ are never in the same $S_k$-orbit; and
     \end{enumerate}
\item bijections $\Phi_0, \Phi_1, \Phi_2, ...$ where each $\Phi_k: \Q_1(k) \to \Q_2(k)$ is such that
    \[
    T_2 \circ \Phi_k\left(E_1(k,(\alpha, v))\right) = \Phi_k \circ S_k\left(E_1(k,(\alpha, v))\right)
    \]
    for all $\alpha \in \{1, ..., t(k)\}$ and all $v \in [h_2(n_k)-1]$.
\end{enumerate}

We next claim that no $x$ belongs to infinitely many $R(k)$.  
We first show this is true for the point $x_0$. Based on the third part of Step 1 of the induction procedure, 
we know $x_0$ always belongs to $(\Qtilde_1(k))_0$.
Since  $x_0\in A_{0,k}$ for every $k$, it is never swapped in Step 2 of the induction procedure.
It is possible that for some $k$, $S_{k-1}(x_0)$ is part of a set swapped in Step 2 of the induction procedure, necessitating that 
$S_k(x_0) \neq S_{k-1}(x_0)$, i.e. $x_0\in R(k)$.  If this is the case, then the set that contained $S_{k-1}(x_0)$ 
would have been swapped with a set disjoint from $A_{0,k}$  meaning $S_k(x_0)\notin A_{0,k}$ and thus $S_k(x_0)\notin A_{0,k'}$ for every $k'\ge k$.  This would mean $S_k(x_0) \notin F(k')$ for every $k'>k$ and so $x_0\notin R(k')$ for every $k'>k$.

We next note that $x_2$ belongs to no $R(k)$.  This is because $S_k(x_2)$ is not even defined for any $k$, since $x_2$ is in the top of every 
$\Qtilde_1(k)$ after the third part of Step 1 of the induction.

Last, we consider $x\notin \{x_0,x_2\}$.
Then for some $k$, $x \notin A_{0,k} \bigcup A_{2,k}$.  At the $k^{th}$ induction stage and beyond, if $x \notin \partial \Qtilde_1(k)$, $x$ would never be swapped into $\partial \Qtilde_1(k)$ (since it isn't in $A_{0,k}$ or $A_{2,k}$).  
If $x \in \partial \Qtilde_1(k)$ at the $k^{th}$ induction stage, $x$ would be swapped out of its location at that time (since it isn't in $A_{0,k}$ or $A_{2,k}$), but then for $l > k$, $x \notin A_{0,l} \cup A_{0,l}$ so $x$ would never be swapped back into $\partial \Qtilde_1(l)$.  
Thus $x \notin F(l)$ for every $l > k$.  Therefore, the only way $x \in R(l)$ is if $S_{l-1}(x) \in F(l)$ for some $l > k$.  
But this can only happen once, analagous to the situation of $x_0$.
Therefore there is an $L$ so that $x \notin R(l)$ for $l > L$.  

Since no $x$ belongs to infinitely many $R(k)$, and because for $x \notin R(k)$, $S_{k-1}(x) = S_k(x)$, and since the boundaries of $\Q_1(l)$ shrink to the two exceptional points $\{x_0, x_2\}$ mentioned earlier, we see that for every $x$ other than  $x_2$, we can define $S : X - \{x_2\} \to X - \{x_0\}$ by
\[
S(x) = \bigcap_{j=0}^\infty \bigcup_{k=j}^\infty S_k(x).
\]
Extend the definition of $S$ to the exceptional point $x_2$ by defining $S(x_2) = x_0$. By the way these points were originally chosen, we know that this specially defined $S$ is of the form $\T_1^\p$ for $\p: X_1 \to \C$.  Thus we obtain a $\C$-speedup $S$ of $(X_1, \T_1)$ defined on the entirety of $X_1$.

Now, for any point $x \in X$, $x$ is only swapped at finitely many induction steps.  Suppose $x$ is not swapped after the $k^{th}$ induction step; since $\Q_1(l)$ refine and generate, every $x \in X_1$ is determined by the sequence of atoms of $\Q_1(l)$ (starting with $l = k+1$) to which it belongs.   Call those atoms $\Q_1(l)(x)$.  Then, define 
\[
\Phi(x) = \bigcap_{l=k+1}^\infty \Phi_l(\Q_1(l)(x));
\] since the atoms of $\Q_2(l)$ also refine and generate, $\Phi$ is a homeomorphism from $X_1$ to $X_2$.  By property (3) above, $\Phi \circ S = T_2 \circ \Phi$, meaning that we have a $\C$-speedup $S$ of $(X_1, \T_1)$ conjugate to $(X_2, T_2)$ as wanted.

\end{proof}

There are two directions in which one might hope to generalize the result of Theorem \ref{homeoimpliesspeedup}. We end this section by discussing the challenges involved.

First, ideally one would be able to conclude an analogous result for arbitrary minimal $\Z^d$-Cantor systems ($X_1, \T_1)$ (as opposed to just $\Z^d$-odometers).  A major problem here is that one needs an analogue of Lemma 3.1, which says that given two subsets of equal measure, there is a function $\p$ taking values in $\C$ such that  $\T_1^\p$ homeomorphically maps one subset to the other.  When $d=1$, such a result exists (see Lemma 3.16 of \cite{Ash}).  But the proof uses induced transformations (first return maps to subsets of the phase space), which are not well-defined for minimal actions of $\Z^d$ when $d \geq 2$.

Second, we conjecture that our Theorem \ref{homeoimpliesspeedup} is valid even if $\T_2$ is a higher-dimensional odometer action.  The problem here is that while K-R partitions exist for $\Z^d$-odometers (and indeed, for all minimal $\Z^d$-Cantor systems \cite{For}),  the boundaries of the K-R partitions decrease to an uncountable collection of points (as opposed to our situation in Theorem \ref{homeoimpliesspeedup}, where we have just the two exceptional points $x_0$ and $x_2$).  Our argument does not allow for this, as we have no method of defining one or more generators of the speedup $\S$ on this uncountable boundary.

% =====================================================================
%
%  SECTION 5:
%     BOUNDED SPEEDUPS
%
% =====================================================================

\section{Bounded speedups of odometers}

We now turn our attention to speedups where the speedup cocycle is bounded.

\begin{defn} Suppose $\S : \Z^{d_2} \actson X$ is a speedup of $\T : \Z^{d_1} \actson X$ with speedup cocycle $\p$.  We say the speedup is \textbf{\emph{bounded}} if for each $\v \in \Z^{d_2}$, the set $\{\p(x, \v) : x \in X\}$ is a bounded subset of $\Z^{d_1}$.
 \end{defn}
 
First, we note the connection between boundedness of a speedup and continuity of the speedup cocycle:

%
% LEMMA 5.2
%   Bounded is same as continuous
%

\begin{lemma} 
Let $\T : \Z^{d_1}\actson X$ be a $\Z^d$-Cantor system and suppose $\S : \Z^{d_2} \actson X$ is a speedup of $\T$.  The following are equivalent:
\begin{enumerate}
\item The speedup is bounded.
\item For each $j \in \{1, ..., d_2\}$,the sets $\{\p(x, \e_j) : x \in X\}$ are bounded. 
\item For each $j \in \{1, ..., d_2\}$, the function $x \mapsto \p(x, \e_j)$ is continuous.
\item For each $\v \in \Z^{d_2}$, the function $x \mapsto \p(x, \v)$ is continuous.
\item The speedup cocycle $\p : X \times \Z^{d_2} \to \Z^{d_1}$ is continuous.
\end{enumerate}
\end{lemma}

\begin{proof} 
It is obvious that (1) implies (2); the fact that (2) implies (1) follows from the cocycle equation; and clearly conditions (3), (4) and (5) are equivalent.  We have (3) implies (2),  because under any continuous function, $X$ is mapped to a compact, hence bounded, subset of the codomain.  

Finally, to show that (2) implies (3), fix $j$ and note that if $\{\p(x, \e_j) : x \in X\}$ is bounded, then there is a finite set $\{\p_1, ..., \p_n\} \subseteq \Z^d$ so that $X = \stackrel[i=1]{n}{\bigsqcup} \{x \in X : \p(x, \e_j) = \p_i\}$.  
By the proof of Theorem \ref{thm1.3}, each set in this union is closed; since the union is finite, each set is also open, making $x \mapsto \p(x,\e_j)$ continuous as wanted.
\end{proof}

In \cite{AAO}, the authors prove that a $\Z$-action that is a minimal bounded speedup of a $\Z$-odometer is an odometer which is conjugate to the original odometer.  We prove below in Theorem \ref{bddspeedupthm}  that a minimal bounded speedup of a $\Z^d$-odometer is an odometer.  
In light of Theorem \ref{mainodometertheorem}, this speedup must be orbit equivalent to the original odometer. However, we prove 
in Theorem \ref{notconjugate} that a bounded speedup of a $\Z^d$-odometer is not necessarily isomorphic, and thus not necessarily conjugate, to the original system, even if it is also an action of $\Z^d$.

% 
%   THEOREM 5.3  (bddspeedupthm)
%      Bounded speedup of an odometer is an odometer
% 

\begin{thm} \label{bddspeedupthm} Let $(X_\gothG, \sigma_\gothG)$ be a free $\Z^{d_1}$-odometer and suppose $\S : \Z^{d_2} \actson X$ is a minimal speedup of $\sigma_\gothG$, where the speedup cocycle $\mathbf{p}$ is bounded.   Then $(X_\gothG,\S)$ is a free $\Z^{d_2}$-odometer.
\end{thm}

\begin{proof}  We first find a decreasing sequence of subgroups $\mathfrak{H} = \{H_{1}, H_2, H_3, ...\}$ with appropriate properties, so that by Definition \ref{CortezDef} and Theorem \ref{basicodomprops}, $(X_{\mathfrak{H}}, \sigma_\mathfrak{H})$ is a free, minimal $\Z^{d_2}$-odometer.  The proof is then concluded by establishing that $(X_{\mathfrak{H}}, \sigma_\mathfrak{H})$ is conjugate to $(X_\gothG,\S)$. 

To find the subgroups making up $\mathfrak{H}$, we start by considering
$\p : X_\gothG \times \Z^{d_2} \to \Z^{d_1}$, the speedup cocycle generating $\S$ from $\sigma_\gothG$.  As $X_\gothG$ is compact and for each $\v \in \Z^{d_2}$ the map $x \mapsto \p(x,\v)$ is continuous, each such map must also be uniformly continuous.

Now consider the sequence $\{\mathcal{P}_j\}$ of partitions of $X_\gothG$ as described in Theorem \ref{KRpartthm}.  As the maximum diameter of any atom of $\mathcal{P}_j$ tends to zero as $j$ increases, there is $J$ such that for all $j \geq J$, whenever $x$ and $y$ lie in the same atom of $\mathcal{P}_j$, $\p(x,\v) = \p(y,\v)$ for all $\v \in \Z^{d_2}$.  So without any loss of generality, we can assume $J = 1$ by renaming our sequence $\gothG$ from $\{G_1, G_2, ...\}$ to $\{G_J, G_{J+1}, G_{J+2}, ...\}$.  This produces a conjugate version of the original odometer $\sigma_\gothG$ by Lemma 1 of \cite{C}. 

Recalling from the proof of Theorem \ref{KRpartthm} that $\mathcal{P}_j = \{B(j,\v) : \v \in [\mathbf{m}_j]\}$ and noting that 
$\sigma_\gothG^\w (B(j,\v)) = B(j, \x)$ where $\x \in [\mathbf{m}_j]$ is congruent to $\v+\w  \mod G_j$, we see that for each $\w \in \Z^{d_1}$,
$\sigma_\gothG^\w$ takes atoms of $\P_j$ to atoms of $\P_j$.  Since  $\p$ is constant on elements of $\mathcal{P}_j$, 
it follows that $\sigma_\gothG^{ \p(x,\v)} = \S^\v(x)$ also takes atoms of $\P_j$ to atoms of $\P_j$ for each $\v \in \Z^{d_2}$.  
In other words, for each $\v \in \Z^{d_2}$ and any atom $A$ of $\mathcal{P}_j$, 
$\S^\v(A)$ is also an atom of $\mathcal{P}_j$.  
As each $\S^\v$ is a homeomorphism, this means that for every $j$, each $\S^\v$ simply permutes the elements of $\mathcal{P}_j$.  
Letting $E_j$ denote the  atom of $\mathcal{P}_j$ containing the identity element $0$ of the group $X_\gothG$, we can then define 
\[
H_j = \{\mathbf{h} \in \Z^{d_2} : \mathbf{S}^\mathbf{h}(E_j) = E_j\}.
\]
It is clear that each $H_j$ is a subgroup of $\Z^{d_2}$. \\

\noindent \emph{Claim 1:} For all $j$, $H_{j} \geq H_{j+1}$.  

Suppose $\mathbf{h} \in H_{j+1}$.  This means $\S^\h(E_{j+1}) = E_{j+1}$.  Since $E_{j+1} \subseteq E_j$, it follows that $\S^\h(E_j) \cap E_j \supseteq E_{j+1} \neq \emptyset$.  But from our earlier observation that each $\S^\v$ permutes the atoms of $\P_j$, we can conclude that $\S^\h(E_j) = E_j$ and so $\h \in H_j$ as desired. \\

\noindent \emph{Claim 2:} $\bigcap_j H_j = \{\mathbf{0}\}$.  

Let $\mathbf{w} \in \bigcap_{j} H_j$.  Then $\S^{\mathbf{w}}(E_j) = E_j$ for all $j$, which means that $\S^{\mathbf{w}}(0) = 0$.  
If $\mathbf{w} \neq \0$, this means that $\sigma_\gothG^\v(0) = 0$ for some nonzero vector $\v \in \Z^{d_1}$, contradicting the freeness of $\sigma_\gothG$. \\

\noindent \emph{Claim 3:} $H_j$ is a finite-index subgroup of $\Z^{d_2}$.  

We need to show there are only finitely many cosets of $H_j$.  We know there are 
only finitely many atoms in $\P_j$: denote them by $A_0=E_j$, $A_1$, ... , $A_k$. Since $S$ is minimal, we can find 
$\v_0$, $\v_1$, ..., $\v_k$ such that $\S^{\v_i}(E_j) = A_i$.  
Our claim will then be proved by showing that the cosets of $H_j$ are exactly $\{ \v_i + H_j : 0\leq i\leq k\}$.
So let $\w\in\Z^{d_2}$: we want to show that $\w$ is in one of $\{ \v_i + H_j : 0\leq i\leq k\}$.
Since $\S^{\w}$ permutes the elements of $\P_j$, 
we know $S^{\w}(E_j)$ is one of those atoms of $\P_j$, say $A_i$.  We also know that $S^{-\v_i}(A_i)=E_j$.  Thus $S^{-\v_i + \w}(E_j)=E_j$, 
meaning that $-\v_i +\w \in H_j$, i.e. $\w \in \v_i + H_j$ as wanted.  \\

In light of Claims 1, 2, and 3, we can conclude that $(X_\mathfrak{H}, \sigma_\mathfrak{H})$ is a free $\Z^{d_2}$-odometer.  What remains is to find $\phi: X_\gothG \rightarrow X_\mathfrak{H}$ with 
$\phi \circ \S^{\v} =  \sigma^{\v}_\mathfrak{H} \circ \phi$ for each $\v\in \Z^{d_2}$,
 which we will build out of a sequence of maps  $\phi_j : \P_j \to \Z^{d_2} / H_j$.  
 To define $\phi_j$, let $A$ be an atom in $\P_j$ and set 
 \[
 H_j(E_j,A) = \{\mathbf{h} \in \Z^{d_2} : \mathbf{S}^\mathbf{h}(E_j) = A\}.
 \]
 We then set $\phi_j(A) = \h + H_j$ where $\h\in  H_j(E_j,A)$.  This map is well-defined because of the following:\\
 
\emph{Claim 4:}   Fix $A$ to be an atom in $\P_j$ and $\h_1$ and $\h_2$ to be any two elements in $H_j(E_j,A)$.  
Then $\h_1 + H_j = \h_2 + H_j$.

We need to show that $\h_1 - \h_2 \in H_j$.  But since  $\mathbf{S}^\mathbf{h_1}(E_j) = A$ and  $\mathbf{S}^\mathbf{h_2}(E_j) = A$ by the definition of $ H_j(E_j,A) $, we immediately see that $\mathbf{S}^\mathbf{h_1 - h_2}(E_j) = E_j$, which gives us the result.  \\

The map $\phi_j$ respects the $\Z^{d_2}$ action in the following way: \\

\emph{Claim 5:}  For all $j\in\mathbb{N}$, all $A \in \P_j$ and all $\v \in \Z^{d_2}$, $\phi_j(\S^\v(A)) = \phi_j(A) + \v$ (where the addition is modulo $H_j$).

Let $\w\in \phi_j(\S^\v(A))$.  So $\w \in \h + H_j$ where $\h\in H_j(E_j,\mathbf{S}^\mathbf{v}(A))$, meaning $\mathbf{S}^\mathbf{h}(E_j) =\mathbf{S}^\mathbf{v}(A)$.  But then $\mathbf{S}^\mathbf{h-v}(E_j)=A$ so $\h- \v \in H_j(E_j,A)$.
On the other hand, $\w \in \h + H_j$ tells us that $\w - \v \in \h - \v + H_j$ and thus $\w - \v \in  \phi_j(A)$ or $\w \in \v +  \phi_j(A)$, as wanted. 

For the other direction, let $\w \in \v +  \phi_j(A) = \v + (\h + H_j)$ for some $\h\in H_j(E_j,A)$.  But if $\mathbf{S}^\mathbf{h}(E_j) = A$ then 
$\mathbf{S}^\mathbf{h+v}(E_j) = \mathbf{S}^\mathbf{v}(A)$ and we have $\h + \v \in H_j(E_j,\mathbf{S}^\mathbf{v}A)$.  This tells us that 
$\w\in (\v + \h) + H_j = \phi_j(\S^\v(A)) $, as wanted.
 \\
 
 We are now ready to define $\phi : X_\gothG \to X_{\mathfrak{H}}$.
 Given $x \in X_\gothG$, let $\mathcal{P}_j(x)$ be the atom of $\P_j$ containing $x$.  Note that $x$ is determined by this decreasing sequence of atoms of $\P_j$.  We can then define  
\[\phi(x) = (\phi_1(\P_1(x)), \, \phi_2(\P_2(x)), \, ...\, ).\]  
The remaining claims show that $\phi$ is the wanted conjugacy.
\\

\emph{Claim 6:} $\phi(x)\in X_{\mathfrak{H}}$.

By definition, $\phi_j(\P_j(x))\in \Z^{d_2} / H_j$ for all $j$, so we need only show that for each $j$, 
\[
q_j(\phi_{j+1}(\P_{j+1}(x)) = \phi_j(\P_j(x)),\]
where we recall that $q_j: \Z^{d_2}/H_{j+1} \to \Z^{d_2}/H_j$ is the quotient map.
In other words, if we denote $\phi_{j+1}(\P_{j+1}(x))$ by $\k + H_{j+1}$ and $\phi_j(\P_j(x))$ by $\h + H_j$, we just need to show that 
$\k \equiv \h$ mod $H_j$.  To do this, note that $E_{j+1}\subseteq E_j$ and $\P_{j+1}(x) \subseteq \P_j(x)$.  
Since by definition of $\phi_{j+1}$ we have that $\S^{\k}(E_{j+1}) = \P_{j+1}(x)$, this means $\S^{\k}(E_j)\cap \P_j(x) \neq \emptyset$.  Since 
$\S^{\k}$ simply permutes the elements of $\P_j$, we have $\S^{\k}(E_j) = \P_j(x)$.  We also know by definition of $\phi_j$ that 
$\h\in H_j(E_j, \P_j(x))$ or $\S^{\h}(E_j) = \P_j(x)$.  Thus $\S^{\k - \h}(E_j) = E_j$ and $\k - \h \in H_j$, as needed.
\\

\emph{Claim 7:} $\phi$ is injective. 

We first show that each $\phi_j$ is injective.  Let $A,B\in \P_j$.  Denote $\phi_j(A)$ by $\h_1 + H_j$ and 
$\phi_j(B)$ by $\h_2 + H_j$, so $\S^{\h_1}(E_j)=A$ and  $\S^{\h_2}(E_j)=B$.  If $\phi_j(A) = \phi_j(B)$ then we must have 
$\h_1 - \h_2 \in H_j$ or $\S^{\h_1 - \h_2}(E_j) = E_j$.  We can rewrite this as $\S^{-\h_2 + \h_1}(E_j) = \S^{-\h_2}(A) =  E_j$ or 
$\S^{\h_2}(E_j) = A$.  But then we have that $A=B$ and thus $\phi$ is injective.

We now consider $x,y\in X_\gothG$.  If $\phi(x) = \phi(y)$ then we have $\phi_j(\P_j(x)) = \phi_j(\P_j(y))$ for every $j$.  
By the above this means $\P_j(x) = \P_j(y)$ for every $j$, which can only happen if $x=y$.
\\

\emph{Claim 8:} $\phi$ is surjective. 

We first show that each $\phi_j$ is surjective.  Let $\h + H_j \in \Z^{d_2} / H_j$.  We know $\S^{\h}(E_j)$ is some atom of $P_j$: call it $A$.  
Thus $\h \in H_j(E_j,A)$ and $\phi_j(A) = \h + H_j$, as wanted.

Now consider $(x_1, x_2, ... )\in X_{\mathfrak{H}}$.  Since $x_j\in \Z^{d_2} / H_j$, we can write it as $x_j = \h_j + H_j$ and,
by the above, find $A_j$ such that $\phi_j(A_j) = \h_j + H_j$.  Since we have $q_j(x_{j+1}) = x_j$ for every $j$ by assumption, 
we then have $\h_{j+1} \equiv \h_j$ mod $H_j$ and so $\S^{\h_{j+1}}(E_j)=A_j$ as well as $\S^{\h_{j+1}}(E_{j+1})=A_{j+1}$.  But we know that 
$E_{j+1}\subseteq E_j$ and thus we must have $A_{j+1}\subseteq A_j$, meaning $\{ A_j \}$ is a decreasing sequence of nested atoms.  
By Theorem \ref{KRpartthm} there is exactly one point $x\in \cap_j A_j$ and we then have $\phi(x) = (x_1, x_2, ... )$ as wanted.
\\

\emph{Claim 9:} $\phi$ and $\phi^{-1}$ are continuous.

This follows because the image and pre-image of any cylinder set is a cylinder set.
\\

\emph{Claim 10:} $\phi$ intertwines the actions of $\S$ and $\sigma_\mathfrak{H}$.

If $x$ is an element of $A\in \P_j$, then $\S^{\v}(x)$ is an element of $\S^{\v}(A)$ and we have that $\P_j(\S^{\v}(x)) = \S^{\v}(\P_j(x))$.
Thus 
\begin{align*}
\phi (\S^{\v}(x)) & =  ( \phi_1(\P_1 (\S^{\v}(x))),  \, \phi_2(\P_2 (\S^{\v}(x))), \, ...\, ) \\
& = (\phi_1(\S^{\v}(\P_1(x))),  \, \phi_2 (\S^{\v}(\P_2(x))),\,  ...\, ) \\
&= (\v + \phi_1(\P_1(x)), \, \v + \phi_2(\P_2(x)), \, ...\, ) \\
& \qquad \textrm{ (by applying Claim 5)} \\
& = \sigma_\mathfrak{H}^{\v} (\phi_1(\P_1(x)), \, \phi_2(\P_2(x),\,  ...\, ) \\
& =  \sigma_\mathfrak{H}^{\v}(\phi(x)).
\end{align*} 
\\
Thus $\phi$ gives a conjugacy between $(X_{\mathfrak{H}}, \sigma_\mathfrak{H})$ and $(X_\gothG,\S)$, making $\S$ an odometer as wanted.

 \end{proof}

Our work yields the following result in dimension $1$, first proven by Alvin, Ash and Ormes \cite{AAO} using combinatorial methods:

% 
%   COROLLARY 5.4
%      Bounded speedup of a Z-odometer is a conjugate Z-odometer
% 

\begin{cor} Let $(X,T)$ be a $\Z$-odometer and suppose $S$ is a minimal speedup of $T$, where the speedup function $p$ is bounded.  Then $(X,S)$ is a $\Z$-odometer, conjugate to $(X,T)$.
\end{cor}
\begin{proof} By Theorem \ref{bddspeedupthm}, $(X,S)$ is a $\Z$-odometer.  The result then follows from Corollary \ref{odometercorollary}.
\end{proof}

\vspace{.1in}

Despite the fact that the bounded speedup of an odometer is an odometer, in the higher-dimensional case the speedup is not necessarily conjugate, nor even isomorphic, to the original odometer.  We show this in the following theorem where we find a $\Z^2$-odometer and construct a bounded speedup of it that is not isomorphic to the original odometer.

% 
%   THEOREM 5.5   (notconjugate)
%      Speedup of an odometer doesn't have to be isomorphic to original
% 

\begin{thm} 
\label{notconjugate} 
There exists a pair of nonisomorphic $\Z^2$-odometers, for which one is a speedup of the other.
\end{thm}

\begin{proof} 
To describe the first odometer, set $G_j = 3^j \Z \times 2^j\Z$ for each $j \geq 1$.  As described in Theorem \ref{basicodomprops}, the sequence $\gothG = \{G_1, G_2, G_3, ...\}$ then yields a minimal $\Z^2$-odometer $(X_\gothG, \sigma_\gothG)$.  
The second odometer will be a speedup of $(X_\gothG, \sigma_\gothG)$, using the partition $\P_1$ as described in Theorem \ref{KRpartthm}  to define the speedup cocycles.  Note that $\P_1$ in this case consists of six sets associated to the six cosets in $\Z^2/G_1$.  We then define $\p_1, \p_2 : X \to \Z^2$ by setting $\p_1(x) = (1,0)$ and \[
\p_2(x) = \left\{\begin{array}{cl}
(0,1) & \textrm{ if $x \in \pi_1^{-1}\left(\{(0,0) + G_1, (1,0) + G_1, (2,0) + G_1\}\right)$} \\
(1,1) & \textrm{ if $x \in \pi_1^{-1}\left(\{(0,1) + G_1, (1,1) + G_1, (2,1) + G_1\}\right)$}.
 \end{array}
\right.\]
The $\Z^2$-action $(X_\gothG,\S)$ generated by $\p_1$ and $\p_2$ is therefore a bounded $\C$-speedup of 
 $\sigma_{\gothG}$ for any cone $\C$ containing the positive $x$-axis and the portion of the line $y=x$ lying in the first quadrant; 
 by Theorem \ref{bddspeedupthm}, $(X_\gothG, \S)$ is a $\Z^2$-odometer.  
 Note that $(X_\gothG, \sigma_\gothG)$ and $(X_\gothG, \S)$ are continuously orbit equivalent  using the identity map and the 
 speedup cocycle $\p$ for the orbit cocycle.  Therefore the orbits of $(X_\gothG, \sigma_\gothG)$ are exactly the same sets as the orbits of $(X_\gothG, \S)$, so since $(X_\gothG, \sigma_\gothG)$ is minimal, $(X_\gothG, \S)$ is minimal as well.

It remains to show that $(X_\gothG, \sigma_\gothG)$ and $(X_\gothG, \S)$ are not isomorphic. We will do this by computing the first cohomology group of each odometer and then using Theorem \ref{isomorphismtest}.  For $H(\sigma_{\gothG})$, note that  
$G_j = 3^j \Z \times 2^j\Z$ has $G_j^* = \frac 1{3^j}\Z \times \frac 1{2^j}\Z$ and thus 
$H(\sigma_{\gothG}) = \stackrel[j=1]{\infty}{\cup}  
G_j^* = \Z[\frac 13] \times \Z[\frac 12]$.  To find $H(S)$, we will follow the construction in the proof of Theorem \ref{bddspeedupthm} to find a sequence of subgroups $ \{G'_1, G'_2, G'_3, ...\}$ associated to $(X_\gothG, S)$ and then similarly compute 
$H(S) = \stackrel[j=1]{\infty}{\cup}  (G'_j)^*$.

Beginning with $G'_1$, recall it is defined to be $\{ \h\in\Z^2: S^{\h}(E_1)=E_1\}$, where $E_1$ is the atom of $\P_1$ containing the identity element of $X_\gothG$.  
We can see that $(3,0)$ and $(2,2)$ are in  $G'_1$ by definition of $\p_1$ and $\p_2$ and then check that these in fact generate all of $G'_1$.
Thus 
\[G'_1 =  \left(\begin{array}{cc} 
3 & 2 \\ 
0 & 2 
\end{array} \right)\Z^2.\] 
In a similar fashion, we can show that 
\[G'_j = \left(\begin{array}{cc} 
3^j & 3^j - 2^{j-1} \\
0 & 2^j
\end{array}
\right)\Z^2.\]
We can therefore calculate
\begin{align*}
(G_1')^* & = \{ (p,q)\in \R^2: (p,q)\cdot (a,b) \in \Z \text{ for all } (a,b)\in G'_1 \} \\
& = \{ (p,q)\in \R^2: (p,q)\cdot (3,0) \in \Z \text{ and } (p,q)\cdot (2,2)\in \Z \} \\
& = \{ (p,q)\in \R^2: \left(\begin{array}{cc} 
                           3 & 0 \\ 
                           2 & 2 
                              \end{array} \right) 
                          \left( \begin{array}{c} p \\ q \end{array}\right) \in \Z^2 \} \\
& = \{ (p,q)\in \R^2: \text{ there exists}(m,n)\in\Z^2 \text{ with }                          
                   \left( \begin{array}{c} p \\ q \end{array}\right)  
                   = 
                   \left(\begin{array}{cc} 
                           3 & 0 \\ 
                           2 & 2 
                              \end{array} \right)^{-1} 
                    \left( \begin{array}{c} m \\ n \end{array}\right)  \}  \\
& =     \left(\begin{array}{rc} 
                           1/3 & 0 \\ 
                           -1/3 & 1/2
                              \end{array} \right)\Z^2 \, = \, \frac{1}{6} \left(\begin{array}{rc} 
                           2 & 0 \\ 
                           -2 & 3
                              \end{array} \right)\Z^2.                
\end{align*}
In a similar fashion, we can show that 
\begin{align*}
(G_j')^* & = \frac{1}{6^j} \left(\begin{array}{cc} 
                           2^j & 0 \\ 
                           2^{j-1}-3^j & 3^j
                              \end{array} \right)\Z^2 \\
                &= \left\{ \left( \begin{array}{c}
                           \frac{1}{3^j} m \\ (\frac{1}{6\cdot 3^{j-1}} - \frac{1}{2^j} ) m + \frac{1}{2^j} n 
                           \end{array}\right): m,n\in \Z \right\} 
\end{align*}
But if $(x,y)\in (G_j')^*$, then $x = \frac{1}{3^j} m$ and $y = (\frac{1}{2} - (\frac{3}{2})^j)x + \frac{1}{2^j}n$
or $y-1/2 x = \frac{1}{2^j}n - \frac{3^j}{2^j}x = \frac{1}{2^j} n - \frac{1}{2^j} m$.
Thus 
\[
H(\S) = \bigcup_{j=1}^\infty (G_j')^* = \left\{(x,y) \in \Z^2 : x \in \Z\left[\frac 13\right], y - \frac 12 x \in \Z\left[\frac 12\right]\right\}.
\]   
We then see that  $H(\S)$ contains $\left(\frac 13, \frac 16\right)$ but $H(\sigma_\gothG)$ does not.  
So by (1) of Theorem \ref{isomorphismtest},  $(X_\gothG, \sigma_\gothG)$ and $(X_\gothG, \S)$ are not conjugate.  

To show that they are not isomorphic, we use contradiction.  Suppose they are isomorphic: then by (2) of Theorem \ref{isomorphismtest},
there is $\alpha \in GL_2(\Z)$ such that $\alpha(H(\sigma_\gothG)) = H(\S)$.  Denote $\alpha$ as 
$\left(\begin{array}{cc} a & b \\ c & d \end{array}\right) \in GL_2(\Z)$.  Since $(0,1/2^j) \in H(\sigma_\gothG)$ for every $j$, we must have 
\[
\left(\begin{array}{cc} a & b \\ c & d \end{array}\right)
\left( \begin{array}{c} 0 \\ 1/2^j\end{array}\right)
= 
\left( \begin{array}{c} b/2^j \\ d/2^j\end{array}\right)
\in H(\S)
\]
for every $j$.  In other words, $b/2^j \in \Z[\frac{1}{3}]$ for every $j$, which implies $b=0$.
We also know $(1/3^j,0) \in H(\sigma_\gothG)$ for every $j$, and thus we must have 
\[
\left(\begin{array}{cc} a & 0 \\ c & d \end{array}\right)
\left( \begin{array}{c} 1/3^j\\ 0 \end{array}\right)
= 
\left( \begin{array}{c} a/3^j \\ c/3^j\end{array}\right)
\in H(\S)
\]
for every $j$.  This means $\frac{c}{3^j} - \frac{1}{2}\, \frac{1}{3^j} = \frac{1}{3^j} (c-\frac{1}{2}a) \in \Z[\frac{1}{2}]$ for every $j$, and thus $c=\frac{1}{2} a$.  But this means the determinant of matrix $\alpha$ is 
\[
\det \left(\begin{array}{cc} 2c & 0 \\ c & d \end{array}\right) = 2cd.
\]
But $2cd$ cannot equal $\pm $1 for any $c,d\in\Z$, contradicting our assumption that $\alpha\in  GL_2(\Z)$.  \end{proof}

As mentioned in the above proof, the odometers $(X_\gothG,\sigma_\gothG)$ and $(X_\gothH,\sigma_\gothH)$ are continuously orbit equivalent even though they are not isomorphic. We conjecture that this is always true, at least in the $d=2$ case, and leave it as an open problem:

% 
%  CONJECTURE 5.6 (bddspeedupctsoe)
%      
% 

\begin{conj} \label{bddspeedupctsoe} Let $(X_\gothG,\sigma_\gothG)$ be a free $\Z^2$-odometer and suppose $\S : \Z^2 \actson X$ is a minimal speedup of $\sigma_\gothG$, where the speedup cocycle $\mathbf{p}$ is bounded.   Then the $\Z^2$-odometer $(X_\gothG,\S)$ is continuously orbit equivalent to $(X_\gothG,\sigma_\gothG)$.
\end{conj}

The result in the other direction is almost immediate:

% 
%  THEOREM 5.7 
%      If Z^2 odometers continuously orbit equivalent, then they are speedup equivalent for some cone C
% 

\begin{thm} 
\label{orbteqvandspeedups}
Let $(X_\gothG,\sigma_\gothG)$ and and $(X_\gothH,\sigma_\gothH)$ be $\Z^d$-odometers that are continuously orbit equivalent.  Then there exists a cone $\C\subseteq \Z^d$ such that $(X_\gothH,\sigma_\gothH)$ is conjugate to a $\C$-speedup of $(X_\gothG,\sigma_\gothG)$.
\end{thm}

\begin{proof}
The definition of continuous orbit equivalence gives us a homeomorphism $\Phi : X_\gothG \to X_\gothH$ and a continuous orbit
cocycle $\h$ such that $\Phi\, \circ\, \sigma_\gothG^{\h(x,\v)}(x) = \sigma_\gothH^\v \circ \Phi(x)$ for all $x \in X$ and $\v \in \Z^d$.  Since $\h$ is continuous and $X_\gothG$ compact, we see that each function $x \mapsto \h(x, \e_j)$ can only take finitely many values as $x$ ranges over $X_\gothG$. Let $\C$ be any cone that contains all the values of the $\h(x, \e_j)$.  We can then define a speedup cocycle by $\p(x,\v) = \h(x,\v)$: the homeomorphism $\Phi$ then gives a conjugacy between the $\C$-speedup so defined and  $(X_\gothH,\sigma_\gothH)$.  
\end{proof}

The above theorem assures us that given two continuously orbit equivalent odometers, there is a cone $\C$ (in fact, many cones) such that one odometer is a bounded $\C$-speedup of the other.  But does this remain true for an arbitrary cone $\C$?  The answer can be no, even if the odometers are assumed conjugate.  We prove this in Corollary \ref{nospeedup}  via an example constructed in Theorem \ref{nospeedup}.  We also show in Theorem  \ref{notrigid}  that there are odometers for which the shape of the cone $\C$ is not an obstruction. We first need this preliminary lemma:

% 
%   LEMMA 5.8 (groupstructure)
%      Characterizes subgroups of Z^2 with appropriate index and product structure
% 

\begin{lemma}
\label{groupstructure}
Let $G < \mathbb{Z}^2$ be a group with index $[\mathbb{Z}^2:G]=6^j$ for some integer $j$.  If 
there exist $m,\widetilde{m}\in\mathbb{N}$ with 
\[  
 3^{m}\Z \times 2^{m}\Z \leq G \leq 3^{\widetilde{m}}\Z \times 2^{\widetilde{m}}\Z,
\]
then $G = 3^{j}\Z \times 2^{j}\Z$.
\end{lemma}

\begin{proof}
This result follows from investigating the form of the elements in $G$.  We begin by noting that $(3^m,0)\in 3^{m}\Z \times 2^{m}\Z \le G$.  Let $x$ be the smallest positive integer such that $(x,0)\in G$ and note $x\in 3^{\widetilde{m}}\Z$, i.e. $x = 3^{\widetilde{m}}z$ for $z\in \Z$.  We then have that the group generated by the two vectors $(3^m,0)$ and $(x,0)$ is a subgroup of $G$.  Looking just at the first coordinates of the elements in this subgroup, note that these coordinates must be generated by the greatest common divisor of $x$ and $3^m$, which thus will have the form $3^{k}$ for some integer $k$ (later, we will show $k=j$).  So we have $(3^k,0)\in G$ and by definition of $x$, it must in fact be that $x = 3^k$.

We next define $b\in \Z$ to be the smallest positive number such that there exists a nonnegative integer $a$ with $(a,b)\in G$.  We may assume $a<3^k$, else subtract a multiple of $(3^k,0)$ until it is so.

We claim that $G$ is generated by $(3^k,0)$ and $(a,b)$.  To show this, let $g\in G$ and write $g = (g_1, g_2)$.  If $g_2=0$, then $g_1$ must be a multiple of $3^k$ and thus $(g_1,g_2)$ is in the group generated by $(3^k,0)$ and $(a,b)$, as wanted.
So assume $g_2 \neq 0$.  Note that $g_2$ must be a multiple of $b$, else we can add to $(g_1, g_2)$ multiples of $(a,b)$ and $(0,3^k)$ until the second coordinate is between 0 and $b$ and the first coordinate is between $0$ and $3^k$, contradicting our definition of $b$.  So we have $g_2 = h b$ for $h\in \Z$.  
We can then say that $(g_1,g_2) - h(a,b) = (g_1-ha, 0) \in G$, which means $3^k$ must divide $g_1-ha$, or $g_1 = ha + c 3^k$ for some integer $c$.  We thus have that $g = (g_1, g_2) = (ha + c3^k, hb) = h(a,b) + c(3^k,0)$, as wanted.

This tells us that $G = \left( \begin{array}{cc}  3^k & a \\
                                       0  & b  \end{array}\right) \Z^2$.  What remains is to show is that $j = k$, $b=2^j$ and $a=0$.
                                       
To investigate the form of $b$, we note that since the index $[\mathbb{Z}^2:G]$ is the determinant of the matrix, it must be that $3^kb = 6^j$.  We thus have that $b = 2^j3^{j-k}$.  On the other hand, we know that $(0,2^m)\in 3^{m}\Z \times 2^{m}\Z \le G$ and thus there must be 
integers $p$ and $q$ such that $(0,2^m) = p(3^k,0) + q(a,b) = (p\, 3^k + q\, a, q\, b)$ or $2^m = q\, b$.  
Thus $b=2^j3^{j-k}$ must divide $2^m$.  This implies that $j=k$ and $b=2^j$, yielding
\[G = \left( \begin{array}{cc}  3^j & a \\
                                       0  & 2^j  \end{array}\right) \Z^2.\]
 
 Finally, we consider $a$.  Since $3^{m}\Z \times 2^{m}\Z \leq G$, we know that $(3^m, 2^m)\in G$.  Thus there must be integers $z$ and $w$ such that  
\[
(3^m, 2^m) = z(3^j,0) + w(a,2^j) = (z3^j+wa, \, w2^j).
\]
 This tells us that $2^m = w2^j$ and $3^m = z3^j+wa$.  The first says that $w$ is a power of 2; together with the second equality we then have that $a$ must be a multiple of $3$.
 But we know that $3^k=3^j$ divides $3^m$ and thus must divide $wa$, hence divide $a$ as well.  If $a$ is nonzero, it follows that $a = 3^l$ for some $l\ge k$.  But we assumed $a<3^k$, and thus we must have $a=0$, as wanted.
\end{proof}
%\vspace{.1 in}

With this technical result in hand, we can show that in some cases, a $\C$-speedup must have a very specific structure:

% 
%   THEOREM 5.9 (nospeedup)
%      There is no bdd speedup of x2, x3 odometer conjugate to itself for certain cones C
% 

\begin{thm} \label{nospeedupnew}  For each $j$, define $G_j = 3^j\Z \times 2^j \Z$ and let $(X_\gothG, \sigma_\gothG)$ be the $\Z^2$-odometer given by $\gothG = \{G_1, G_2, G_3, ...\}$.  Let $\C = \{0,1,2,...\} \times \{0,1,2,...\} - \{(0,0)\}$.

Suppose $(X_\gothG, \S)$ is a bounded $\C$-speedup of $(X_\gothG, \sigma_\gothG)$ which is conjugate to $(X_\gothG, \sigma_\gothG)$.  Then the speedup is ``a product of two one-dimensional speedups'', meaning that $\S^{\e_1}$ is a speedup of $\sigma_\gothG^{\e_1}$  and $\S^{\e_2}$  is a speedup of $\sigma_\gothG^{\e_2}$.
\end{thm}

\begin{proof}  Suppose $\S$ is a bounded $\C$-speedup of $\sigma_\gothG$ that is conjugate to $\sigma_\gothG$, and let $\p$ be its speedup cocycle.  We will prove the theorem by showing that $\p(x, \e_1) = p_1(x)\e_1$ and $\p(x, \e_2) = p_2(x)\e_2$ for functions $p_1, p_2 : X_\gothG \to \Z^+$. 

We begin by considering some of the structure of an odometer and its speedup.  For instance, let $\P_j$ be the sequence of K-R partitions for $(X_\gothG, \sigma_\gothG)$ as described in Theorem \ref{KRpartthm}. 
Just as was done in Theorem \ref{bddspeedupthm}, we can let $J$ be the smallest $j \geq 0$ such that for each $\v \in \Z^2$, $\p(x,\v)$ is constant on the atoms of $\P_J$ (and hence constant on all atoms of any $\mathcal{P}_j$ for any $j \geq J$).  Thus we can think of 
${ \sigma_\gothG}^{\p(\cdot,\v)} = \S^{\v}$ as 
simply permuting the $6^j$ elements of $\mathcal{P}_j$.

Also similar to what we did in Theorem \ref{bddspeedupthm},
 let $H_j = \{\v \in \Z^2 : \S^\v(E_j) = E_j\}$, where $E_j$ is the atom of $\P_j$ containing the zero element of $X_\gothG$.  Recall this yields a sequence of groups with  $H_J \geq H_{J+1} \geq ...$ which can be used to define (a conjugate version of) the odometer 
$(X_\gothG,\S)$.

Choose any $x \in X_\gothG$.  For each $j \geq J$, let $A_j$ be the atom of $\P_j$ containing $x$.  Observe that it follows from the justification of Claim 3 in the proof of Theorem \ref{bddspeedupthm} that for any $j \geq J$, 
\[H_j = \{\v \in \Z^2 : \S^\v(A_j) = A_j\}.\]

Next, since we are assuming $(X_\gothG,\S)$ and $(X_\gothG, \sigma_\gothG)$ are conjugate, we can use Lemma 1 of \cite{C}  to say that for each $j$, there exists $m_j$ and $\widetilde{m}_j$ such that 
\[
   3^{m_j}\Z \times 2^{m_j}\Z \leq H_j \leq 3^{\widetilde{m}_j}\Z \times 2^{\widetilde{m}_j}\Z. 
\]
By Lemma \ref{groupstructure}, we then know there is some $k_j\in\mathbb{N}$ such that $H_j = 3^{k_j}\Z \times 2^{k_j}\Z$.  In particular, this says that $(3^{k_J},0) \in H_J$ and so $\S^{(3^{k_J},0)}$ sends $A_J$  to itself.  We can rewrite this as saying that 
${ \sigma_\gothG}^{\p(\cdot,(3^{k_J},0))} $ sends $A_J$ to $A_J$.  But we said above that $\p(\cdot,(3^{k_J},0))$ is constant on $A_J$  and thus we can find a constant vector $(\alpha_J, \beta_J)\in \mathbb{Z}^2$ which equals $\p(\cdot,(3^{k_J}, x))$ when restricted to $A_J$.  That is, 
$\S^{(3^{k_J},0)} = { \sigma_\gothG}^{(\alpha_J, \beta_J)} $ when restricted to $A_J$.

Our goal is to show $\beta_J =0$, and our main tool will be the fact that by the definition of $\sigma_\gothG$, we know that for  all $j$,
%\color{black},
\[
\{\v \in \Z^2 : \sigma_\gothG^\v(A_j) = A_j\} = 3^j\Z \times 2^j\Z.
\] 
We begin by noting that when $j=J$, the above tells us that $(\alpha_J, \beta_J) = (3^Ja, 2^Jb)$ for some nonnegative integers $a$ and $b$.

We next consider $j=J+1$.  We similarly find that $\S^{(3^{k_{J+1}},0)}$ sends $A_{J+1}$ to $A_{J+1}$. But for all points in $A_{J+1}\subset A_J$,
\begin{align*}
\S^{(3^{k_{J+1}},0)} & = \left(\, \S^{(3^{k_{J}},0)} \,\right)^{3^{k_{J+1}-k_J}} \\
& = \left(\, { \sigma_\gothG}^{(3^Ja, 2^Jb)} \,\right)^{3^{k_{J+1}-k_J}} \\
& = \sigma_\gothG^{(3^{k_{J+1}-k_J}3^Ja, 3^{k_{J+1}-k_J}2^Jb)}.
\end{align*}
We then know that $\sigma_\gothG^{(3^{k_{J+1}-k_J}3^Ja, 3^{k_{J+1}-k_J}2^Jb)}$ sends $A_{J+1}$ to $A_{J+1}$.  But we also know that the only vectors $\v \in \Z^2$ such that $\sigma_\gothG^{\v}$ sends $A_{J+1}$ to $A_{J+1}$ are those in $3^{J+1}\Z \times 2^{J+1}\Z$, and thus we have that 
\[
(3^{k_{J+1}-k_J}3^Ja, 3^{k_{J+1}-k_J}2^Jb) = (3^{J+1} m, 2^{J+1} n) \textrm{ for some }m, n \in \Z.
\]
In particular, this says that $3^{k_{J+1}-k_J}2^Jb = 2^{J+1} n$, implying that $b$ must be divisible by $2$.

We can then repeat this argument for each $j>J$.  In other words, we have $\S^{(3^{k_{j}},0)}$ sends $A_{j}$ to $A_{j}$ and, for points in $A_j$, we can rewrite  $\S^{(3^{k_{j}},0)}$ as 
\begin{align*}
\S^{(3^{k_{j}},0)} & = \left(\, \S^{(3^{k_{J}},0)} \,\right)^{3^{k_{j}-k_J}} \\
& = \left(\, { \sigma_\gothG}^{(3^Ja, 2^Jb)} \,\right)^{3^{k_{j}-k_J}} \\
& = \sigma_\gothG^{(3^{k_{j}-k_J}3^Ja, 3^{k_{j}-k_J}2^Jb)}.
\end{align*}
We then again use that the only vectors $\v \in \Z^2$ such that $\sigma_\gothG^{\v}$ sends  $A_j$ to $A_j$ must be of the form $3^{j}\Z \times 2^{j}\Z$ to conclude that
\[
(3^{k_{j}-k_J}3^Ja, 3^{k_{j}-k_J}2^Jb) = (3^{j} p, 2^{j} q) \textrm{ for some }p, q \in \Z.
\]  
In particular, this says that $3^{k_{j}-k_J}2^Jb = 2^{j} q$, implying that $b$ must be divisible by $2^{j-J}$.  As this holds for all  $j>J$, $b$ must be zero.  \\

We now proceed with proving $\p(x, \e_1) = (p_1(x),0)$ for every $x \in X_\gothG$.  Since $b = 0$, the vector  $(\alpha_J, \beta_J)$ found above actually can be written  $(\alpha_J,0)$ and  thus,
\begin{equation} \label{xyz}
\S^{(3^{k_J},0)}(x) = \sigma_\gothG^{(\alpha_J, \beta_J)}(x) = \sigma_\gothG^{(\alpha_J,0)}(x).
\end{equation}
We have now shown that for every $x\in X_\gothG$,   $\p(x, 3^{k_J}\e_1) = (\alpha_J,0)$ for some $\alpha_J = \alpha_J(x) \in \Z^+$.

Next, we see from the cocycle equation that
%\begin{align*}
\begin{equation}\label{cocycleeqn3J}
\p(x, 3^{k_J}\e_1) 
 = \p(x,\e_1) + \p(\S^{\e_1}x, \e_1) + \p(\S^{2\e_1}x, \e_1) + ... + \p(\S^{(3^{k_J}-1)\e_1}x, \e_1).
\end{equation}
%\end{align*}
From the choice of $\C$, each vector in this sum has a second coordinate which is non-negative, so if any of these vectors have a positive second coordinate, then $\p(x, 3^{k_J}\e_1) = (\alpha_J,0)$ must also have a positive second coordinate, a contradiction.  In particular, this says that the second coordinate of $ \p(x,\e_1)$ must be zero and we can write 
 $\p(x, \e_1) = (p_1(x),0)$ for some function $p_1 : X_\gothG \to \Z^+$, as wanted.  \\
 
A similar argument shows that there is a function $p_2 : X_\gothG \to \Z^+$ such that $\p(x, \e_2) = (0,p_2(x))$ for all $x \in X_\gothG$.  In particular, we can write $\S^{(0,2^{k_J})}$ as $\sigma_\gothG^{(\gamma_J, \delta_J)}$ where $(\gamma_J, \delta_J) = (3^Jc, 2^Jd)$.  We can prove $c = 0$, and therefore $\gamma_J = 0$, by observing that whenever $j > J$, $\S^{(0,2^{k_j})}$ sends each atom of each $\P_j$ to itself, which implies $c$ is divisible by $3^{j-J}$ for all $j > J$.   \end{proof}

% 
%   COROLLARY 5.10  (nospeedup)
%      There is no bdd speedup of x2, x3 odometer conjugate to itself for certain cones C
% 

The following corollary then says that the form of the cone $\C$ can severely restrict what a bounded $\C$-speedup can be conjugate to.

\begin{cor} \label{nospeedup} For each $j$, define $G_j = 3^j\Z \times 2^j \Z$ and let $(X_\gothG, \sigma_\gothG)$ be the $\Z^2$-odometer given by $\gothG = \{G_1, G_2, G_3, ...\}$.  Let $\C \subseteq \Z^2$ be a cone.  If there is a bounded $\C$-speedup of $(X_\gothG,\sigma_\gothG)$ which is conjugate to  $(X_\gothG,\sigma_\gothG)$, then $\C$ must contain points along both coordinate axes.
\end{cor}
\begin{proof} Notice from the proof of Theorem \ref{nospeedupnew} that equation (\ref{xyz}) can be obtained irrespective of the cone $\C$, i.e. the speedup cocycle $\p$ defining $\S$ must satisfy $\p(x, 3^{k_J}\e_1) = (\alpha_J(x),0)$ for all $x \in X_\gothG$.  
If $\C$ contains no points along the horizontal axis, then by definition of it being a cone, it must lie entirely above or entirely below the horizontal axis.
Thus the second coordinate of  $\p(\cdot, \e_1)$ will always be positive or will always be negative.
By using equation (\ref{cocycleeqn3J}),  this in turn tells us that the second coordinate of $\p(x, 3^{k_J}\e_1)$ cannot be zero, contrary to what we said above.  Thus we know that $\C$ must indeed contain some points on the horizontal axis.
%So any $\C$ containing the values of $\p(\cdot, \e_1)$ must contain points along the horizontal axis.  For if not, then the second coordinates of the values of $\p(\cdot, \e_1)$ would either be always positive, or always negative, and from the cocycle equation, (\ref{xyz}) would be impossible.  
Similarly, we can obtain $\p(x, 2^{k_J}\e_2) = (0,\delta_J(x))$ for all $x \in X_\gothG$; since $\C$ contains the range of $\p(\cdot, \e_2)$, it must also contain points along the vertical axis.
\end{proof}

Thus for the odometer $\sigma_\gothG$ given by the groups $G_j = 3^j \Z \times 2^j \Z$,  there is substantial rigidity in the speedups of $\sigma_\gothG$ which are conjugate to $\sigma_\gothG$.  In the next example, we demonstrate that this rigidity does not always exist:

%
%  THEOREM 5.11  (notrigid)
%      There is bdd speedup of x2, x2 odometer conjugate to itself for any cone C
%

\begin{thm} \label{notrigid} Let $\mathcal{G} = \{G_n\}_{n=1}^\infty$, where $G_n = 2^n \Z \times 2^n\Z$.  Then, for any cone $\C \subseteq \Z^2$, there is a $\C$-speedup of the $\Z^2$-odometer $(X_{\mathcal{G}}, \sigma_{\mathcal{G}})$ which is conjugate to $(X_{\mathcal{G}}, \sigma_{\mathcal{G}})$.
\end{thm}

\begin{proof} Given $\C \subseteq \Z^2$, choose an integer vector $\tilde{\p} = (p_1, p_2)$ such that $\tilde{\p} \in \C$, $\tilde{\p} + (0,1) \in \C$, and $p_{1}$ is odd.  We use this $\tilde{\p}$ to define the speedup of $(X_{\mathcal{G}}, \sigma_{\mathcal{G}})$:  define the cocycle 
$\p(x, (v_1, v_2)) = v_1\tilde{\p} + v_2(\tilde{\p} + (0,1))$ for all $x \in X_{\mathcal{G}}$, and let  $(X_\mathcal{G}, \S)$ be the $\C$-speedup of $(X_\mathcal{G}, \sigma_\mathcal{G})$ given by cocycle $\p$. \\

\noindent \emph{Claim 1:}  $\S$ is minimal. 

Consider the refining and generating sequence of partitions $\P_n$ described in Theorem \ref{KRpartthm}. We will show that the $\S$-orbit of the identity element $0\in X_{\mathcal{G}} $ intersects every atom of $\P_n$ for every $n$.  So fix $n$ and an atom of $\P_n$: such an atom must be the inverse image of some coset $(u_1, u_2) + G_n$ under the map $\pi_n$, as defined in Section \ref{odometers}.
We will know the orbit of $0$ intersects this atom if we can find $\v\in\Z^2$ such that the $n^{th}$ coordinate of $\S^{\v}(0)$ equals $(u_1, u_2) $ mod $G_n = 2^n \Z \times 2^n\Z$.

Since $p_1$ is odd, we can find an integer $m$ such that $m p_1 = u_1$ mod $2^n$.  With this $m$ fixed, we can then pick an integer $k$ such that
$(m p_2 + m) +k = u_2$ mod $2^n$.  Now let $v_1 = -k$ and $v_2 = k+m$.  Then
\begin{align*}
\S^{(v_1,v_2)}(0) & =  \sigma_{\mathcal{G}}^{\p(0, (v_1, v_2))}(0) \\
& = (v_1\tilde{\p} + v_2(\tilde{\p} + (0,1)) \textrm{ mod }G_1, v_1\tilde{\p} + v_2(\tilde{\p} + (0,1))   \textrm{ mod }G_2, ...).
\end{align*}
The $n^{th}$ coordinate is thus  
\begin{align*}
v_1\tilde{\p} + v_2(\tilde{\p} + (0,1)) & =  (v_1 p_1 + v_2 p_1, v_1 p_2 + v_2 (p_2+1) ) \\
                                                         & =  (m p_1, k + m p_2 + m ) \\
                                                         & = (u_1, u_2) \,\,{\rm{ mod }}\,\, 2^n \Z \times 2^n\Z,
\end{align*} 
as wanted.

Because the partitions $\P_n$  generate the topology of $X_\mathcal{G}$, the above tells us that the $\S$-orbit of $0$ must be dense in $X_{\mathcal{G}}$.
Since the $\S$-orbit of any $x \in X_\mathcal{G}$ is just the $\S$-orbit of $0$ translated by $x$, it follows that every $\S$-orbit is dense, meaning $\S$ is minimal. \\

By Theorem \ref{bddspeedupthm}, $(X_\mathcal{G}, \S)$ is therefore a $\Z^2$-odometer and is thus associated to some decreasing sequence $\{G'_n\}$ of finite-index subgroups of $\Z^2$.  One such sequence can be obtained by following the construction in the proof of Theorem \ref{bddspeedupthm}, where we found that   $G'_n = \{\h \in \Z^2 : \S^\h(E_n) = E_n\}$, where $E_n$ is the atom of $\P_n$ containing $0$.\\

\noindent \emph{Claim 2:} For every $n$ there are positive integers $a_n$ and $b_n$ such that $G'_n = 2^{a_n}\Z \times 2^{b_n}\Z$.\\
We prove this by induction.  Consider $G'_1 = \{\h \in \Z^2 : \S^\h(E_1) = E_1\}$.  Note that $(2,0)\in G'_1$ because 
$$S^{(2,0)} =  \sigma_{\mathcal{G}}^{\p(\cdot,(2,0))} =   \sigma_{\mathcal{G}}^{2\tilde{\p}} =  \sigma_{\mathcal{G}}^{(2p_1,2p_2)}$$ and 
$(2p_1,2p_2) = (0,0)$ mod $2 \Z \times 2\Z$.  
Similarly we have $(0,2)\in G'_1$ and thus $2 \Z \times 2\Z \subseteq G'_1$.
If they are not equal, we can find $(a,b) \in G'_1$ with at least one of $a,b$ odd.  By subtracting off elements from $2 \Z \times 2\Z$ we can then find $(\tilde{a},\tilde{b})  \in G'_1$ with one or both of $\tilde{a},\tilde{b}$ equal to 1.
But note that $(1,0)\notin G'_1$, since 
$$S^{(1,0)} =  \sigma_{\mathcal{G}}^{\p(\cdot,(1,0))} =  \sigma_{\mathcal{G}}^{\tilde{\p}} =  \sigma_{\mathcal{G}}^{(p_1,p_2)}$$ and $p_1$ is odd. We similarly see that $(0,1)\notin G'_1$.  This leaves $(1,1)$ as the only possibility, yet
$$S^{(1,1)} =  \sigma_{\mathcal{G}}^{\p(\cdot,(1,1))} =  \sigma_{\mathcal{G}}^{2\tilde{\p}+(0,1)} = 
 \sigma_{\mathcal{G}}^{(2p_1,2p_2+1)}.$$  But $2p_2+1\neq 0 \, \mod 2$, and thus $ \sigma_{\mathcal{G}}^{(2p_1,2p_2+1)}$ cannot send $E_1$ to $E_1$. \\
 
 Now assume $G'_n = 2^{a_n}\Z \times 2^{b_n}\Z$ for nonnegative integers $a_n$ and $b_n$.  We will in fact show that 
$G'_{n+1} = 2^{a_{n+1}}\Z \times 2^{b_{n+1}}\Z$, where $a_{n+1} \in \{a_n, a_n + 1\}$ and $b_n \in \{b_n, b_n + 1\}$.  To do so, it suffices to show that if 
$(w_1, w_2) \in G'_n$, then  $(2w_1, 2w_2) \in G'_{n+1}$.  Toward that end, suppose $(w_1, w_2) \in G'_n$.  Then 
\begin{align*}
\S^{(w_1,w_2)}(0)  = \sigma_\gothG^{  w_1{\tilde{\p} } + w_2(  {\tilde{\p}}+(0,1) )  }(0) \in E_n.
\end{align*}
Note that when the partition $\P_n$ is refined into $\P_{n+1}$, the atom $E_n$ of $\P_n$ is subdivided into four atoms of $\P_{n+1}$, namely $\pi^{-1}((0,0) + G_{n+1})$, $\pi^{-1}((2^{n},0) +G_{n+1})$, $\pi^{-1}((0,2^{n})+G_{n+1})$ and $\pi^{-1}((2^{n}, 2^{n}) + G_{n+1})$.  Thus $\S^{(w_1,w_2)}(0)$ lies in exactly one of these four atoms.  No matter which of these atoms contains $\S^{(w_1,w_2)}(0)$, it must be the case that $\S^{(2w_1,2w_2)}(0) \in \pi^{-1}((0,0) + G_{n+1}) = E_{n+1}$, making $(2w_1, 2w_2) \in G'_{n+1}$ as wanted.  \\

At this point, we have $G'_n = 2^{a_n} \Z \times 2^{b_n}\Z$ for sequences $\{a_n\}, \{b_n\}$ of integers satisfying $a_1 = b_1 = 1$ and $a_{n+1} - a_n \in \{0,1\}$, $b_{n+1}-b_n \in \{0,1\}$ for all $n$.  \\

\noindent \emph{Claim 3:}   The sequences $a_n \to \infty$ and $b_n \to \infty$. 

We prove this by contradiction.  So assume one of these, say $a_n$, does not diverge to infinity. That means there is $N$ such that $a_n = a_N$ for all $n \geq N$.  So for all $n \geq N$, $G'_n = 2^{a_N}\Z \times 2^{b_n}\Z$, which implies $(2^{a_N},0) \in G'_n$.  By the definition of $G'_n$, this means
\[
\S^{(2^{a_N},0)} (0) = \sigma_{\mathcal{G}}^{2^{a_N}\tilde{\p}}(0) \in E_n
\]
for all $n \geq N$, and it follows that $\sigma_{\mathcal{G}}^{  2^{a_N}\tilde{\p}  }(0) = 0$.  This contradicts the freeness of $\sigma_{\mathcal{G}}$.  Therefore $a_n \to \infty$, and a similar proof shows $b_n \to \infty$. \\

Since $(X_\mathcal{G}, \S)$ is conjugate to $(X_{\mathcal{G}'}, \sigma_{\mathcal{G}'})$, where $\mathcal{G}' = \{G'_1, G'_2, ...\}$ and each $G'_n$ is of the form $ 2^{a_n}\Z \times 2^{b_n}\Z$ with $a_n, b_n \to \infty$, we see that
\[
H(\S) = \bigcup_{n=1}^\infty (G'_n)^* = \Z\left[\frac 12\right] \times \Z\left[\frac 12\right] = H(\sigma_{\mathcal{G}}).
\]
By Theorem \ref{isomorphismtest}, $(X_{\mathcal{G}}, \S)$ is conjugate to $(X_\mathcal{G}, \sigma_\mathcal{G})$.
\end{proof}

Theorems \ref{nospeedupnew} and \ref{notrigid}  show a role that the particular structure of the speedup and the choice of cone $\C$ play in the theory of bounded $\C$-speedups of higher-dimensional odometers.  In dimension 1, a bounded minimal speedup of an odometer is a conjugate odometer \cite{AAO}.  Furthermore, there are only two cones in $\Z$, namely the positive integers and negative integers.  Since any $\Z$-odometer $\sigma_G$ is conjugate to its inverse $\sigma_\gothG^{-1}$,  the structure of speedups whose cocycle takes values in the positive cone is identical to the structure of speedups whose cocycle takes negative values.

But in higher dimensions, there are many cones one could be interested in, and depending on the odometer being sped up, the choice of cone $\C$ can play a substantial role in the structure of its bounded, minimal $\C$-speedups. For example, while any resulting speedup must be an odometer, no matter the cone (see Theorem \ref{bddspeedupthm}), whether or not said speedup can be conjugate to the original may be completely independent of $\C$, exemplified in Theorem \ref{notrigid}, or highly dependent on the choice of $\C$, as demonstrated in Corollary \ref{nospeedup}.

\bibliographystyle{amsplain}

\end{document}